\documentclass[12pt]{amsart}
\usepackage{graphicx,amssymb}
\usepackage[usenames]{color}
\usepackage{amsthm,amsfonts,amsmath,amssymb,latexsym,epsfig,mathrsfs,yfonts,marvosym,latexsym,epsfig}
\usepackage[all]{xy}
\AtEndDocument{\vfill\eject\batchmode}
\DeclareMathAlphabet\oldmathcal{OMS}        {cmsy}{b}{n}
\SetMathAlphabet    \oldmathcal{normal}{OMS}{cmsy}{m}{n}
\DeclareMathAlphabet\oldmathbcal{OMS}       {cmsy}{b}{n}
\usepackage{eucal}

\newtheorem{theorem}{Theorem}[section]
\newtheorem{lemma}[theorem]{Lemma}
\newtheorem{proposition}[theorem]{Proposition}
\newtheorem{corollary}[theorem]{Corollary}

\newtheorem{def/prop}[theorem]{Definition/Proposition}

\newtheorem{definition}[theorem]{Definition}

\theoremstyle{definition}
\newtheorem{remark}[theorem]{Remark}
\newtheorem*{ack}{Acknowledgements}
\newtheorem{example}{Example}[section]

\DeclareSymbolFont{bbold}{U}{bbold}{m}{n}
\DeclareSymbolFontAlphabet{\mathbbold}{bbold}

\def\BOne{\mathchoice{\scalebox{1.16}{$\displaystyle\mathbbold 1$}}{\scalebox{1.16}{$\textstyle\mathbbold 1$}}{\scalebox{1.16}{$\scriptstyle\mathbbold 1$}}{\scalebox{1.16}{$\scriptscriptstyle\mathbbold 1$}}}

\def\fract#1#2{\raise4pt\hbox{$ #1 \atop #2 $}}

\def\bbc{{\mathbb C}}

\def\bbn{{\mathbb N}}

\def\bbp{{\mathbb P}}
\def\bbq{{\mathbb Q}}
\def\bbr{{\mathbb R}}

\def\bbz{{\mathbb Z}}

\def\grl{\lambda}

\def\gro{\omega}

\def\bfa{{\bf a}}

\def\bfk{{\bf k}}

\def\bfw{{\bf w}}
\def\bfx{{\bf x}}

\def\bfz{{\bf z}}

\def\cala{{\mathcal A}}

\def\calc{{\mathcal C}}
\def\cald{{\mathcal D}}
\def\cale{{\mathcal E}}
\def\calf{{\mathcal F}}

\def\calh{{\mathcal H}}
\def\cali{{\mathcal I}}

\def\call{{\mathcal L}}

\def\calo{{\mathcal O}}

\def\calS{{\mathcal S}}

\def\cals{{\oldmathcal S}}

\def\gc{{\mathfrak c}}

\def\ge{{\mathfrak e}}

\def\gh{{\mathfrak h}}

\def\gs{{\mathfrak s}}
\def\gt{{\mathfrak t}}
\def\gu{{\mathfrak u}}

\def\gz{{\mathfrak z}}

\def\gA{{\mathfrak A}}

\def\lra{\longrightarrow}

\def\<{\langle}
\def\>{\rangle}
\def\ra#1{\to}

\def\hook{\mathbin{\hbox to 6pt{%
                 \vrule height0.4pt width5pt depth0pt
                 \kern-.4pt
                 \vrule height6pt width0.4pt depth0pt\hss}}}

\begin{document}

\title{Existence and Non-Existence of Constant Scalar Curvature and Extremal Sasaki Metrics}
\author[C.P. Boyer]{Charles P. Boyer}
\author[H. Huang]{ Hongnian Huang}
\author[E. Legendre]{Eveline Legendre}
\author[C.W. T{\o}nnesen-Friedman]{Christina W. T{\o}nnesen-Friedman}

\address{Charles P. Boyer, Department of Mathematics and Statistics,
University of New Mexico, Albuquerque, NM 87131.}
\email{cboyer@unm.edu} 
\address{Hongnian Huang, Department of Mathematics and Statistics,
University of New Mexico, Albuquerque, NM 87131.}
\email{hnhuang@gmail.com} 
\address{Eveline Legendre\\ Universit\'e Claude Bernard (Lyon 1)\\
Institut Camille Jordan\\ 43 boulevard du 11 novembre 1918
F-69622 Villeurbanne Cedex\\ France}
\email{legendre@math.univ-lyon1.fr}
\address{Christina W. T{\o}nnesen-Friedman, Department of Mathematics, Union
College, Schenectady, New York 12308, USA } \email{tonnesec@union.edu}

\thanks{The first author was partially supported by grant \#519432 from the Simons Foundation. The third author is partially supported by France ANR project BRIDGES No ANR-21-CE40-0017. The fourth author was partially supported by grant \#422410 from the Simons Foundation.}

\date{\today}
\begin{abstract}
We discuss the existence and non-existence of constant scalar curvature, as well as extremal, Sasaki metrics. We prove that the natural Sasaki-Boothby-Wang manifold over the admissible projective bundles over local products\footnote{See Definition \ref{adkahman} for what we mean by ``local product'' here.} of non-negative CSC K\"ahler metrics, as described in \cite{ACGT08}, always has a constant scalar curvature (CSC) Sasaki metric in its Sasaki-Reeb cone.  Moreover, we give examples that show that the extremal Sasaki--Reeb cone, defined as the set of Sasaki--Reeb vector fields admitting a compatible extremal Sasaki metric, is not necessarily connected in the Sasaki--Reeb cone, and it can be empty even in the non-Gorenstein case. We also show by example that a non-empty extremal Sasaki--Reeb cone need not contain a (CSC) Sasaki metric which answers a question posed in \cite{BHLT21}. The paper also contains an appendix where we explore the existence of K\"ahler metrics of constant weighted scalar curvature, as defined in \cite{Lah19}, on admissible manifolds over local products of non-negative CSC K\"ahler metrics.
\end{abstract}
\maketitle

\markboth{CSC Sasaki Metrics}{C.P. Boyer, H. Huang, E. Legendre, and Christina W. T{\o}nnesen-Friedman}

\tableofcontents

\section{Introduction}
In the recent years, powerful tools have been developed to investigate the question of existence of extremal K\"ahler and Sasaki metrics, particularly regarding the more famous cases of Einstein and constant scalar curvature K\"ahler/Sasaki metrics. As expected, these tools, including various notions of stability \cite{CoSz12} and weighted extremal K\"ahler metrics \cite{ApCa18,ApCaLe21}, are now used to explore the extremal problem within families of K\"ahler/Sasaki structures. This is the idea pursued in the present paper where we use weighted extremal K\"ahler metrics together with admissible constructions on projective bundles to describe more precisely some {\it extremal Sasaki--Reeb cones}.

This cone is the Sasakian counterpart of the extremal K\"ahler cone (i.e, on a given a complex manifold, the cone of K\"ahler classes containing extremal K\"ahler metrics) but sits in the Lie algebra $\mathfrak{t}$ of a compact torus $T$ acting on complex cone $(Y,\widehat{J})$. More precisely, for a given $T$--invariant complex cone $(Y,\widehat{J})$, the Sasaki--Reeb cone $\gt^+(Y,\widehat{J})$ of $(Y,\widehat{J})$, is the set of polarizations $\xi\in \mathfrak{t}$ admitting a compatible $T$--invariant radial K\"ahler potential $r\in C^\infty( Y,\bbr^+)$, that is $r$ is proper, $T$--invariant and satisfies $\call_{-\widehat{J}\xi}r=r,\,\; dd^c r>0.$ The Sasaki–Reeb cone admits an equivalent definition in terms of CR-structures induced on the level set of each radial potential as we recall in \S\ref{ss:ExtSas}.

The extremal Sasaki--Reeb cone, say $\ge(Y,\widehat{J})$, is then the subset of polarizations $\xi\in\gt^+(Y,\widehat{J})$ admitting a compatible $T$--invariant radial K\"ahler potential $r\in C^\infty( Y,\bbr^+)$ whose associated Sasaki metric on the level set $\{r=1\}$ is extremal in the sense of \cite{BGS06}. For example, if $(Y,\widehat{J})$ is the cone over a polarized K\"ahler manifold $(X,L)$ admitting an extremal K\"ahler metric in the class $2\pi c_1(L)$ then the regular Sasaki--Reeb vector field (given by the fiberwise $\bbc^*$ action) lies in $\ge(Y,\widehat{J})$.\\ 

Although the Sasaki--Reeb cone is known to be open, strictly convex and polyhedral \cite{Ler02,BGS06,CoSz12,BovC16,ApCaLe21}, very little is known about the extremal Sasaki--Reeb cone. As proved in \cite{BGS06}, transversal scaling of Sasaki structures implies that $\ge(Y,\widehat{J})$ is a subcone of $\gt^+(Y,\widehat{J})$.
Moreover, $\ge(Y,\widehat{J})$ is open, and can be empty or all of $\gt^+(Y,\widehat{J})$. 

One contribution of the present article, is the exhibition of two families of examples of complex cones for which the extremal Sasaki-Reeb cone has interesting properties. \\

\begin{itemize}
\item {\bf Examples}~\ref{nonexistenceex} {\it non-Gorenstein examples,  (i.e. when the polarization is not that
of the canonical or anticanonical bundle), with Sasaki cone of dimension larger than one, where the extremal Sasaki-Reeb cone is empty; such examples were (to the best of our knowledge) previously known only in the Gorenstein case \cite{BovC16}};\\

\item {\bf Examples}~\ref{prob612} {\it showing that the extremal Sasaki-Reeb cone need not be connected}.
\end{itemize}

We have a special interest in constant scalar curvature Sasaki structures (CSCS). We denote $\gc\gs\gc(Y,\widehat{J})$ the set of $\xi \in \gt^+(Y,\widehat{J})$ for which there exists a compatible $T$--invariant radial K\"ahler potential $r\in C^\infty( Y,\bbr^+)$ whose associated Sasaki metric on the level set $\{r=1\}$ has constant scalar curvature. It is known, \cite{BGS06,FOW06} that
\begin{equation}\label{eq:cscINTERS}
    \gc\gs\gc(Y,\widehat{J})= \ge(Y,\widehat{J})\cap \{\xi \in \gt^+(Y,\widehat{J})\,|\, \mbox{Fut}_\xi=0\}
\end{equation}where $\mbox{Fut}_\xi$ is the transversal (or Sasaki) Futaki invariant. Note that $\{\xi \in \gt^+(Y,\widehat{J})\,|\, \mbox{Fut}_\xi=0\}$ is non-empty for any complex cone over a compact Sasaki manifold, see \cite{BHL17}. Note that this is purely a Sasakian phenomenon. For example, $\bbc\bbp^2$ blown up at a point has a nonvanishing Futaki invariant for all K\"ahler classes. However, we also provide an example here for which the intersection \eqref{eq:cscINTERS} is empty.  \\

\begin{itemize}
\item {\bf  Examples} of Sasaki-Reeb cones with compatible extremal Sasaki metrics (i.e, $\ge(Y,\widehat{J})\neq \emptyset$) but no compatible CSCS metrics (i.e, $\gc\gs\gc(Y,\widehat{J})=\emptyset$). See Theorem \ref{nocscthm} and Example \ref{prob612}.
\end{itemize} This last family of examples answers a question posed in  \cite{BHLT21}.\\

\begin{remark}
Note that the first and third kind of examples above have analogies in K\"ahler geometry. For instance, there are well-known examples by
Levine \cite{Lev85} of K\"ahler manifolds with no extremal metrics whatsoever and none of Calabi's  extremal K\"ahler metrics on the Hirzebruch surfaces \cite{Cal82} are CSC. Note also that Levine's examples have no compact subgroups in their automorphisms group which implies that the Sasaki-Reeb cone associated to any of these examples is only one dimensional (with no extremal representative).  
\end{remark}

All the examples above are built in the context of the Sasakian geometry $(M,\cals)$ as circle bundles over admissible K\"ahler manifolds as defined in \cite{ACGT08}, denoted here by $N^{ad}$. We often refer to this as the Boothby-Wang construction \cite{BoWa58} which occurs in the more general symplectic-contact category.  In particular, our main result is the following condition ensuring that the CSCS cone is non-empty for some of these important cases. 
\begin{theorem}\label{theoIntroExist}[see Theorem \ref{CSCexistence}]
Let $\Omega$ be an admissible Hodge K\"ahler class on the admissible projective bundle $N^{ad} =\bbp(E_0 \oplus E_{\infty}) \lra N$, where $N$ is a compact K\"ahler manifold which is a local product of nonnegative CSCK metrics\footnote{See Definition \ref{adkahman} for what we mean by ``local product'' here.}.  Then for the complex cone $(Y,\widehat{J})$ over $(N^{ad},\Omega)$, the Sasaki-Reeb cone $\gt^+(Y,\widehat{J})$ (relative to the $2$--dimensional torus obtained from the double tower of fibrations) has a non-empty subset $\gc\gs\gc(Y,\widehat{J})$ of compatible radial K\"ahler potentials with associated CSCS metrics.
\end{theorem}

In previous work \cite{BoTo13,BoTo14a,BoTo19a,BoTo20a} the first and last authors have given many examples  proving the existence of constant scalar curvature Sasaki metrics. The method of proofs used the admissible construction of projective bundles developed in \cite{ACGT08} together with the Sasaki join construction \cite{BGO06}. This latter construction, which can be thought of as the Sasaki analogue of de Rham decomposition of K\"ahler manifolds, divides Sasaki manifolds into two types, those which can be described as joins are said to be {\it decomposable}; whereas, those that cannot be described by joins are {\it indecomposable}. For a detailed description see \cite{BHLT16}. 

In the rest of the paper, the examples and results are stated in terms of CR structures and Sasaki metrics. In particular, Theorem \ref{CSCexistence} below is the Sasaki version of Theorem \ref{theoIntroExist}.  Subsection \ref{ss:ExtSas}, and in particular Remark~\ref{rem:ConeToSasaki}, makes the transition between the CR-Sasaki and the complex cone dictionaries.  

The key novel idea in the proof of Theorem \ref{theoIntroExist} (Theorem \ref{CSCexistence}) is to employ the link, discovered by \cite{ApCa18}, between the weighted extremal K\"ahler metrics and Sasaki extremal metrics and then use the properness of the Einstein-Hilbert functional \cite{BHL17} to discover CSCS solutions in the admissible setting.

\begin{ack}
The authors thank Vestislav Apostolov and Ivan Cheltsov for helpful and encouraging correspondence, and the anonymous referee for their careful reading and valuable comments.
\end{ack}

\section{Extremal Sasaki geometry}
\subsection{Brief review on extremal Sasaki structures}\label{ss:ExtSas}
To explain the notation and point of view we are using in this paper, let recall that a Sasakian structure on a manifold $M^{2n+1}$ is is a special type of contact metric structure $\cals=(\xi,\eta,\Phi,g)$ with underlying almost CR structure $(\cald,J)$ where $\eta$ is a contact form such that $\cald=\ker\eta$, $\xi$ is its Reeb vector field, $J=\Phi |_\cald$, and $g=d\eta\circ (\BOne \times\Phi) +\eta\otimes\eta$ is a Riemannian metric. $\cals$ is a Sasakian structure if $\xi$ is a Killing vector field, so we call it a Sasaki-Reeb vector field, and the almost CR structure is integrable, i.e. $(\cald,J)$ is a CR structure. We refer to \cite{BG05} for the fundamentals of Sasaki geometry. We call $(\cald,J)$ a {\it CR structure of Sasaki type}, and $\cald$ a {\it contact structure of Sasaki type}. We shall always assume that the Sasaki manifold $M^{2n+1}$ is compact and connected.

Let $\cals=(\xi,\eta,\Phi,g)$ be a $T$--invariant Sasakian structure where $T$ is a compact torus acting effectively on $M$, that is $T\subset\mbox{Diff}(M)$ and such that $\xi\in \mbox{Lie } T=:\gt$. In this setting, the Sasaki--Reeb cone of $(M,\cals)$, relatively to $T$, is $$\gt^+(\cald,J) := \{\tilde{\xi}\in \gt\,|\, \eta(\tilde{\xi})>0\}.$$ Each $\tilde{\xi}\in \gt^+(\cald,J)$ determines a $T$--invariant Sasakian structure, compatible with $(\cald,J)$ with contact form $\eta(\tilde{\xi})^{-1}\eta$. Therefore, sometimes in this paper, and in the literature, elements of $\gt^+(\cald,J)$ are called Sasakian metrics.

If instead of fixing a contact structure, we fix a Sasaki--Reeb vector field $\xi$ and vary the contact structure, we obtain a contractable space of Sasakian structures (all compatible with $\xi$), namely
\begin{equation}\label{Sasspace}
{\mathcal S}(\xi,\bar{J})=\{\eta+d_B^c\varphi~|~\varphi\in C^\infty_B(M),\,  (\eta+d^c_B\varphi)\wedge(d\eta +dd_B^c\varphi)^n\neq 0\},
\end{equation}
where $d^c_B\varphi= -d\varphi\circ \Phi$. The space ${\mathcal S}(\xi,\bar{J})$ is an infinite dimensional Fr\'echet manifold. When starting with a $T$--invariant Sasakian structure as above, it is convenient to consider only the space of  $T$--invariant isotopic contact structures ${\mathcal S}(\xi,\bar{J})^T$. Note that the Sasaki--Reeb cones, relative to $T$, of elements in ${\mathcal S}(\xi,\bar{J})^T$ all agree.

All the contact structures in ${\mathcal S}(\xi,\bar{J})$ have the same basic cohomology class $[d\eta]_B$ in the basic cohomology group $H^{1,1}(\calf_\xi)$ as described in Section 7.2.2 of \cite{BG05}, and are compatible with the same holomorphic transversal structure $(\xi,\bar{J})$. Each representative $\eta\in {\mathcal S}(\xi,\bar{J})$ determines a transverse K\"ahler structure with transverse K\"ahler metric $g^T=d\eta\circ (\BOne\times \Phi)$. Note that $d\eta$ is not exact as a basic cohomology class, since $\eta$ is not a basic 1-form. We want to search for a `preferred' Sasakian structure $\cals_\varphi$ which represents the basic cohomology class $[d\eta]_B$. This leads to the study \cite{BGS06} of the Calabi functional 
\begin{equation}\label{calabifunct}
\cale_2(g)=\int_Ms^2_gdv_g
\end{equation}
where the variation is taken over the space $\calS(\xi,\bar{J})$. As in the K\"ahler case the Euler-Lagrange equation is a 4th order PDE 
whose critical points are those Sasaki metrics whose transversal scalar curvature $s_g$ is a {\it Killing potential}, that is, its gradient is transversely holomorphic, or, equivalently, the contact vector field of $s_g$ preserves the CR structure. Such Sasaki metrics (structures) are called {\it extremal}. Note that a $T$--invariant Sasakian structure $\cals=(\xi,\eta,\Phi,g)$ with a contact moment map $\mu:M\rightarrow \gt^*$, is extremal if $s_g$ is the pull-back by $\mu$ of an affine-linear function on $\gt^*$. In practice, this is a convenient way to check that a Sasakian structure is extremal.    

In this article we study the following set, that we call the {\it extremal Sasaki--Reeb cone} $$\ge(\cald,J) := \{\xi \in \gt^+(\cald,J)\,|\, \exists \mbox{ an extremal Sasakian structure in } {\mathcal S}(\xi,\bar{J})^T\}.$$ 
\begin{remark}\label{rem:ConeToSasaki}
To recover the definition of the extremal Sasaki--Reeb cone given in the introduction, recall from \cite{ApCaLe21} that ${\mathcal S}(\xi,\bar{J})^T$ is in one to one correspondence with the space of $T$--invariant radial potentials, relative to $\xi$, on the complex cone $(Y,\widehat{J})$ over $M$. The complex structure $\widehat{J}$ on $Y$ depends on $\xi$ and on the transversal holomorphic structure $(\xi,\bar{J})$. However, given two commuting Sasaki--Reeb vector fields $\xi,\tilde{\xi}\in \gt^+(\cald,J)$, their associated complex cones are biholomorphic \cite{HeSu12}. Therefore, $\gt^+(\cald,J)= \gt^+(Y,\widehat{J})$ and $\ge(\cald,J)= \ge(Y,\widehat{J})$ where we identify a Reeb vector field $\xi$ with its lift to $(Y,\widehat{J})$.    
\end{remark}

An important special case of extremal Sasaki structures are those of constant scalar curvature (CSCS) in which case the gradient of $s_g$ is the zero vector field. An extremal Sasakian structure is CSC if and only if the transversal Futaki invariant $Fut_\cals=Fut_\xi:\gh^T(\xi,\bar{J})\lra \bbr$ vanishes. Here $\gh^T(\xi,\bar{J})$ is the Lie algebra of transversally holomorphic vector fields, and we use the alternative definition (see eg. \cite{BHLT15})    
\begin{equation}\label{Futinv}
Fut_\xi(X)= \int_M (s_g- \overline{s_g})\eta(X) \eta\wedge (d\eta)^n
\end{equation} 
with $\overline{s_g}:=\int_M s_g\, \eta(X) \eta\wedge (d\eta)^n/\int_M\eta\wedge (d\eta)^n$.
Note that $Fut_\xi$ depends only on the isotopy class ${\mathcal S}(\xi,\bar{J})$ and not on the element used to represent it.

Thus, as described in the introduction, we get nested cones
$$\gc\gs\gc(\cald,J)\subset \ge(\cald,J)\subset \gt^+(\cald,J)$$ which are the main topics of this paper.

\subsection{The weighted extremal case}
One effective way of treating the indecomposable case alluded to in the introduction is to use the weighted extremal construction developed in \cite{ApCa18,ApCaLe21,ApJuLa21}. It can also be considered as a generalization of \cite{BoTo21} where the projective bundle is a twist one stage 3 Bott orbifold. Here we briefly review the weighted extremal approach. Let $(N,g,\omega)$ be a K\"ahler manifold of complex dimension $m$, $f$ a positive Killing potential on $N$, and (weight) $p\in \bbr$.
Then the  $(f,p)$-Scalar curvature of $g$ is given by
\begin{equation}\label{weightedscal}
Scal_{f,p}(g)= f^2 Scal(g)  -2(p-1) f\Delta_g f - p(p-1)|df|^2_g,
\end{equation}
If $Scal_{f,p}(g)$ is a Killing potential, $g$ is said to be a {\it $(f,p)$-extremal K\"ahler metric}.  
The case $p=2m$ has been studied by several people and is interesting due to the fact that $Scal_{f,2m}(g)$ computes the scalar curvature of the Hermitian metric $h = f^{-2} g$. Note that from time to time we will also simply call an $(f,p)$-extremal K\"ahler metric a {\it weighted extremal metric} (without specifying
$p$ and $f$). The definition of the latter term as well as the definition of {\it weighted scalar curvature} is in fact a bit more general and is introduced by A. Lahdili in \cite{Lah19}. For our purpose the definition in \eqref{weightedscal} (which agrees with (2) in \cite{ApCa18}) shall suffice.

The case of interest to us here is when $p=m+2$. This case is related to the study of extremal Sasaki metrics \cite{ApCa18, ApCaLe21, ApJuLa21}. Indeed, if we assume that the K\"ahler class $[\omega/2\pi]$ is an integer class, then $\frac{Scal_{f,m+2}(g)}{f}$ is equal to the transverse scalar curvature of a certain Sasaki structure (determined by $f$) in the Sasaki-Reeb cone of the Boothby-Wang constructed Sasaki manifold over $(N,g,\omega)$. More precisely, if $\chi$ is the Reeb vector field of the Sasaki structure coming directly from the Boothby-Wang construction over $(N,g,\omega)$ and $f$ is viewed as a pull-back to the Sasaki manifold, then $\xi:= f\chi$ is a Reeb vector field in the Sasaki-Reeb cone giving a new Sasaki structure. While the pull-back from $N$ of $Scal(g)$ is the Tanaka-Webster scalar curvature of the Tanaka-Webster connection induced by $\chi$, the expression $\frac{Scal_{f,m+2}(g)}{f}$ pulls back from $N$ to be the Tanaka-Webster scalar curvature of the Tanaka-Webster connection induced by $\xi$. The latter is then also identified with the transverse scalar curvature of the Sasaki structure defined by $\xi$. This fact is seen from the details of the proof of Lemma 3 in \cite{ApCa18}. As also follows from Theorem 1 of \cite{ApCa18}, the Sasaki structure determined by $f$ is extremal if and only if $g$ is $(f,m+2)$-extremal. In this case, the corresponding Sasaki structure is CSC if and only if $\frac{Scal_{f,m+2}(g)}{f}$ is constant.

\subsection{Admissible K\"ahler manifolds}\label{briefadmintro}
In this paper we study the Sasakian geometry of Boothby-Wang manifolds over the admissible K\"ahler manifolds defined in \cite{ACGT08}. 
Admissible K\"ahler manifolds are very tractable for exploring the existence of K\"ahler metrics with special geometries. The original case in point is Calabi's construction of extremal K\"ahler metrics on Hirzebruch surfaces and related higher dimensional ruled manifolds \cite{Cal82}.
Subsequently these types of constructions have been explored by many (e.g. \cite{Gua95,Hwa94,HwaSi02,Koi90,KoSa86,LeB91b,PePo98,Sim91,To-Fr98}). 
For full details for these manifolds we refer to e.g. \cite{ACGT08} or Sections 2.1 and 2.2 of \cite{ApMaTF18}. Here we will just give a brief introduction:

\begin{definition}\label{adkahman}
Suppose $N$ is a compact K\"ahler manifold and that there exist simply connected K\"ahler manifolds $N_a$ of complex dimension $d_a$ such that $N$ is covered by $\prod_{a\in \cala}N_a$ with $\cala=\{1,\ldots,I\}$. Assume that on each $N_a$ there is an $(1,1)$ form
$\omega_a$, which is a pull-back of a tensor (also denoted by $\omega_a$) on $N$, such that either $\omega_a$ or $-\omega_a$ is a K\"ahler form of a constant scalar curvature K\"ahler (CSCK) metric $g_a$ or $-g_a$. 

We will say that $N$ is ``{\bf covered by a product} of simply connected CSC K\"ahler manifolds, $(N_a,\pm \gro_a, \pm g_a)$'' or that $N$ is ``a {\bf local product} of CSC K\"ahler manifolds, $(N_a,\pm \gro_a, \pm g_a)$.''

Let $E_0$ and $E_\infty$ be projectively flat Hermitian holomorphic vector bundles on $N$ of ranks $d_0+1$ and $d_\infty+1$, respectively, which satisfy the condition
\begin{equation}\label{ast}
\frac{c_1(E_\infty)}{d_\infty+1} -\frac{c_1(E_0)}{d_0+1} = \frac{1}{2\pi}\sum_{a\in \cala}[\gro_{a}].
\tag{\textasteriskcentered}
\end{equation}

Then the projective bundle $N^{ad}=\bbp(E_0\oplus E_\infty)\lra N$ is said to be {\bf admissible}. 
\end{definition}
Note that the complex dimension of $N^{ad}$ is $m=\sum_{a\in{\hat\cala}}d_a+1$. Note also that the Equation \eqref{ast} puts topological constraints on the manifolds $N^{ad}$. In particular, an admissible bundle is never trivial. 

The fibers of such projective bundles admit an $S^1$ action which gives rise to (a choice of a fixed background K\"ahler metric and) a moment map construction $\gz:N^{ad}\rightarrow [-1,1]$ as described in \cite{HwaSi02,ACGT08}.
This gives rise to a family of K\"ahler metrics which are defined on all of $N^{ad}$ and locally described on $N^{ad}_0=\gz^{-1}(-1,1)$ by
\begin{equation}\label{g}
g=\sum_{a\in\hat{\cala}}\frac{1+x_a\gz}{x_a}g_a+\frac {d\gz^2}
{\Theta (\gz)}+\Theta (\gz)\theta^2,\quad
\omega = \sum_{a\in\hat{\cala}}\frac{1+x_a\gz}{x_a}\omega_{a} + d\gz \wedge
\theta,
\end{equation}
where $\theta$ is a connection 1-form,  $\Theta$ is a smooth function with domain containing
$(-1,1)$, and $x_a$ for $a\in \cala$, are real numbers of the same sign as
$g_{a}$ and satisfying $0 < |x_a| < 1$. Here $\hat{\cala}$ is the finite set $\cala=\{1,\ldots,I\}$ if $d_0=d_\infty=0$, and $\hat{\cala}$ equals $\cala\cup \{0\}$ if $d_0>0,d_\infty=0$, equals $\cala\cup \{\infty\}$ if $d_0=0,d_\infty>0$, and equals $\cala\cup \{0\}\cup \{\infty\}$ if $d_0>0,d_\infty>0$. Further, for $d_0\neq 0$, we set $x_0=1$ and let $g_0$ be the Fubini-Study metric on $\bbc\bbp^{d_0}$ with scalar curvature $2d_0(d_0+1)$ (or $2d_0 s_0$ with $s_0=d_0+1$). Similarly, for
$d_\infty\neq 0$, we set $x_\infty=-1$ and let $-g_\infty$ be the Fubini-Study metric on $\bbc\bbp^{d_\infty}$ with scalar curvature $2d_\infty(d_\infty+1)$ (or $-2d_\infty s_\infty$ with $s_\infty=-(d_\infty+1)$). Moreover, $\Theta$ satisfies the following conditions:
\begin{align}
\label{positivity}
(i)\ \Theta(\gz) > 0, \quad -1 < \gz <1,\quad
(ii)\ \Theta(\pm 1) = 0,\quad
(iii)\ \Theta'(\pm 1) = \mp 2.
\end{align}
The triple $(N^{ad},\gro,g)$ where $N^{ad}$ is given in Definition \ref{adkahman} and $(\gro, g)$ in Equation \eqref{g}, is called an {\bf admissible K\"ahler manifold} with admissible K\"ahler metric $g$.

With all else being fixed, the {\bf admissible K\"ahler class} $\Omega=[\omega]$ is determined by the choice of $x_a$ for $a\in \cala$.
If the $x_a$ values are all rational, then, up to scale,  $\Omega$ is a rational cohomology class. In this case, and for convenience, we simply say that $\Omega$ is rational.
Note that a rational admissible K\"ahler class can be rescaled to be an integer class. In general, any rescale of an admissible K\"ahler class will also be called admissible.

\section{Existence and non-existence}\label{E&NE}
In addition to the brief introduction above, we refer the reader to \cite{ACGT08} (or the more recent summary in Sections 2.1 and 2.2 of \cite{ApMaTF18}) for a detailed description of admissible manifolds, their admissible metrics, and admissible K\"ahler classes. Except for a couple of exceptions (pointed out below), we stay faithful to the notation used in \cite{ACGT08} and \cite{ApMaTF18}. In particular, we have a moment map $\gz: N^{ad} \rightarrow [-1,1]$ (denoted by $z$ in  \cite{ACGT08} and \cite{ApMaTF18}) associated to a natural $S^1$-action of the admissible K\"ahler forms. Now, for the case where $N$ is a local product\footnote{See Definition \ref{adkahman} for what we mean by ``local product'' here.} of {\em non-negative} CSCK metrics, by Proposition 11 of \cite{ACGT08} and (the second half of) Theorem 1 in \cite{ApMaTF18} we have that every admissible K\"ahler class admits an admissible $(c\,\gz+1,p)$-extremal metric for any choice of $c\in (-1,1)$ and $p\in \bbr$. 
In more detail, when $c=0$ an $(c\,\gz+1,p)$-extremal metric is just an extremal metric and we appeal to Proposition 11 of \cite{ACGT08}. When $0<|c|<1$, by rescaling we can
let $f=|\gz+{\mathbf a}|$, where ${\mathbf a} = 1/c$ and appeal to (the second half of) Theorem 1 in \cite{ApMaTF18} 
\footnote{Note that the assumption of ${\mathbf a} >1$ (and so $|\gz+{\mathbf a}|= \gz+{\mathbf a}$) in \cite{ApMaTF18} is merely practical and all the arguments are easily adapted to include the ${\mathbf a}<-1$ (and $|\gz+{\mathbf a}|= -\gz-{\mathbf a}$) case as well.}.
Alternatively one can recast Section 2.3 as well as the proof of Theorem 3.1 in \cite{ApMaTF18}, replacing $f=|z+{\mathbf a}|$ with $f=c\,\gz+1$ and reprove the above existence result in a self-contained manner. For simplicity, we shall not summarize this whole process, but only extract exactly what we need for our purposes here.

\subsection{Sufficient conditions for CSC Sasaki metrics in the Sasaki-Reeb cone}\label{cscsect}
Here we will assume that the admissible K\"ahler class $\Omega$ is rational and so there is a Boothby-Wang constructed Sasaki manifold $(M,\cals)$ given by an appropriate re-scale of $\Omega$. The (re-scales of) Reeb vector fields induced by the Killing potentials $c\,\gz+1$, for 
$c\in (-1,1)$, span a 2-dimensional subcone of the Sasaki-Reeb cone. We note that the Sasaki-Reeb cone $\gt^+(\cals)$ of $(M,\cals)$ has dimension $d_0+d_\infty +2 + \dim(\gt)$ where $\gt$ is the Lie algebra of the maximal torus of the automorphism group $\gA\gu\gt(N)$ of the K\"ahler manifold $N$. We also note that $c=0$ corresponds to
the initial ray (since a constant Killing potential $f$ corresponds to a rescale of the initial Reeb vector field coming from the Boothby-Wang construction).
We are ready for

\begin{theorem}\label{CSCexistence}
Suppose $\Omega$ is a rational admissible K\"ahler class on the admissible manifold $N^{ad}=\bbp(E_0 \oplus E_{\infty}) \lra N$, where $N$ is a compact K\"ahler manifold which is a local product\footnote{See Definition \ref{adkahman} for what we mean by ``local product'' here.} of nonnegative CSCK metrics. 
Let $(M,\cals)$ be the Boothby-Wang constructed  Sasaki manifold given by an appropriate rescale of $\Omega$. Then the corresponding Sasaki-Reeb cone will always have a
(possibly irregular) CSC-ray (up to isotopy).
\end{theorem}

\begin{proof}
If the (admissible) extremal metric in $\Omega$ (this is the $c=0$ case) is CSC, we are done. In general we will show that there exist some $c\in (-1,1)$ such that the extremal Sasaki structure determined by $f=c\,\gz+1$ has constant scalar curvature.

Note that the admissible $(c\gz+1,p)$-extremal metric has $(c\, \gz+1,p)$-Scalar curvature $Scal_{c\,\gz +1,p}=A_1\gz+A_2$, where, adapting the details of Section 2.3 of \cite{ApMaTF18}, $A_1$ and $A_2$ are given by the unique solution to the linear system
\begin{equation}\label{A1andA2}
\begin{array}{ccc}
\alpha_{1, -(1+p)} A_1 + \alpha_{0, -(1+p)} A_2 =  2 \beta_{0, (1-p)}\\
\\
\alpha_{2, -(1+p)} A_1 + \alpha_{1, -(1+p)} A_2 =  2\beta_{1, (1-p)},
\end{array}
\end{equation}
where
\begin{equation}\label{alphabeta}
\begin{split}
\alpha_{r,k} =& \int_{-1}^1 (c\,t +1)^{k} t^r p_c(t)dt, \qquad p_c(t)=\prod_{a\in \hat{\cala}}(1+x_at)^{d_a}, \\
\beta_{r,k}   = &  \int_{-1}^1\Big(\sum_{a\in {\hat \cala}} \frac{x_ad_as_a}{1+x_at}\Big)  p_c(t) t^r (c\,t +1)^{k} dt \\
                        & + \big( (-1)^r(1-c)^{k}p_c(-1) + (1+c)^{k} p_c(1)\big),\end{split}
 \end{equation}  
 where $s_a=\pm\frac{Scal_{\pm g_a}}{2d_a}$ is the normalized scalar curvature.                    
Notice that 
\begin{itemize}
\item $c\, \alpha_{r+1,k}+\alpha_{r,k}=\alpha_{r,k+1}$
\item $c\, \beta_{r+1,k}+\beta_{r,k}=\beta_{r,k+1}$.
\end{itemize}

Now, for $p=m+2$, where $m=\sum_{a\in{\hat\cala}}d_a+1$, the extremal Sasaki structure determined by $f=c\,\gz +1$ has constant scalar curvature if and only if 
$Scal_{c\,\gz +1,p}=A_1\gz+A_2$ is a constant multiple of $c\,\gz +1$, i.e., if and only if  $A_1 -c\,A_2=0$. 
Now, $A_1 - c\,A_2=0$ may be re-written as            
$$(\alpha_{1,-(1+p)}\beta_{0,(1-p)}-\alpha_{0,-(1+p)}\beta_{1,(1-p)})-c\,(\alpha_{1,-(1+p)}\beta_{1,(1-p)}-\alpha_{2,-(1+p)}\beta_{0,(1-p)}))=0,$$                      
which in turn is equivalent to $\alpha_{1,-p}\beta_{0,(1-p)}-\alpha_{0,-p}\beta_{1,(1-p)}=0$, which with $p=m+2$ is

\begin{equation}\label{CSCobstruction}
 \alpha_{1,-(m+2)}\beta_{0,-(m+1)}-\alpha_{0,-(m+2)}\beta_{1,-(m+1)}=0.
\end{equation}

Let $K_c$ denote the Reeb vector field in the Sasaki-Reeb cone of $\cals$ induced by $f_{K_c}=c\,\gz+1$.
That is, $\eta^X_\cald (K_c)=c\,\gz+1$ (when $\gz$ is lifted to $\cals$ and $X$ is the initial Reeb vector field of the Boothby-Wang construction alluded to in Theorem \ref{CSCexistence}). In Section \ref{admissibleSFInvariant} we shall see directly that Equation \eqref{CSCobstruction} is an obstruction to the vanishing of the Sasaki-Futaki invariant 
$Fut_{K_c}$. Inspired by this fact and Lemma 3.1 in \cite{BHLT15}, we turn our attention to the {\em Einstein-Hilbert functional}:

Up to an overall positive rescale, the Einstein-Hilbert functional introduced in Section 3 of \cite{BHLT15} is given by
$$H_S(K_c):=H_S(c)=S_c^{m+1}/V_c^m,$$
where $V_c=\alpha_{0,-(m+1)}$ and
$$S_c =  \int_{-1}^1 Scal_{c\,\gz+1,m+2}(g)(c\,\gz+1)^{-(m+2)} p_c(\gz)\,d\gz= 2\beta_{0,-m}.$$
Thus, up to an overall positive rescale, $H_S(c)=\frac{\left(\beta_{0,-m}\right)^{m+1}}{\left(\alpha_{0,-(m+1)}\right)^m}$.
Using that
\begin{itemize}
\item $\frac{d}{dc}\left[\alpha_{0,-(m+1)}\right]=-(m+1)\alpha_{1,-(m+2)}$
\item  $\frac{d}{dc}\left[\beta_{0,-m}\right]=-m\beta_{1,-(m+1)}$,
\end{itemize}
we easily verify that 
$$
\begin{array}{ccl}
H_S'(c) &= &\frac{m(m+1)(\beta_{0,-m})^m (\alpha_{0,-(m+1)})^{m-1}}{(\alpha_{0,-(m+1)})^{2m}}\left[ \alpha_{1,-(m+2)}\beta_{0,-m}-\alpha_{0,-(m+1)}\beta_{1,-(m+1)} \right]\\
\\
&= & \frac{m(m+1)(\beta_{0,-m})^m (\alpha_{0,-(m+1)})^{m-1}}{(\alpha_{0,-(m+1)})^{2m}}\left[ \alpha_{1,-(m+2)}\beta_{0,-(m+1)}-\alpha_{0,-(m+2)}\beta_{1,-(m+1)} \right].
\end{array}
$$
Now $\alpha_{0,-(m+1)}>0$ and $\beta_{0,-m}> 0$ (the latter is due to the fact that we are assuming $N$ is a local product of {\em nonnegative} CSCK metrics) and therefore critical
points of $H_S(c)$ correspond exactly to solutions of Equation \eqref{CSCobstruction} which in turn would give us that an extremal Sasaki structure determined by $f=c\,\gz+1$ that has constant scalar curvature. Showing that $H_S(c)$ must have a least one critical point $c\in(-1,1)$, would finish the proof.

To that end, first notice that $H_S(c)=\frac{\left(\beta_{0,-m}\right)^{m+1}}{\left(\alpha_{0,-(m+1)}\right)^m}$ is a smooth positive function for $c\in (-1,1)$. To show that $H_S(c)$ has a critical point inside $(-1,1)$, we will show that $\displaystyle \lim_{c\rightarrow \pm 1^{\mp}}H_S(c) = +\infty$. Since $c\rightarrow \pm 1^{\mp}$ represents $K_c$ approaching the boundary of the Sasaki cone, this actually follows directly from Lemma 5.1 in \cite{BHL17}, but for the sake of transparency we will verify this directly.

First notice that 
$$\alpha_{0,-(m+1)} =  \int_{-1}^1 (c\,t +1)^{-(m+1)} (1+t)^{d_0} p_0(t)dt =  \int_{-1}^1 (c\,t +1)^{-(m+1)} (1-t)^{d_\infty} p_{\infty}(t)dt,$$
where $p_0(t) = p_c(t)/(1+t)^{d_0}$ and $p_\infty(t) = p_c(t)/(1-t)^{d_\infty}$ are polynomials of degree $m-1-d_0$ and $m-1-d_\infty$, respectively. Keep in mind that
the complex dimension $m=\sum_{a\in{\hat\cala}}d_a+1$ (with $d_a>0$ for $a\in \cala$ and $d_a\geq 0$ for $a\in \hat\cala$) and note that $p_0(-1)>0$ and $p_{\infty}(1)>0$. From the definition of $p_c(t)$ in Equation~\eqref{alphabeta} we also know that 
$p_0(t)$ and $p_\infty(t)$ are positive for $-1<t<1$. By repeated integration by parts, we can see that $\alpha_{0,-(m+1)}$ is a rational function of $c$ which is
positive and has no vertical asymptotes within the interval $(-1,1)$. To understand the behavior near $c=\pm 1$, we consider repeated integration by parts and observe that
$$\alpha_{0,-(m+1)}=\displaystyle \frac{\left( \frac{(m-1-d_0)! \, d_0! \, p_0(-1)}{m! \, c^{d_0+1}} + o(1-c)\right)}{(1-c)^{m-d_0}}=\frac{\left( \frac{(-1)^{d_\infty+1}(m-1-d_\infty)! \, d_\infty! \, p_\infty(1)}{m! \, c^{d_\infty+1}} + o(1+c)\right)}{(1+c)^{m-d_\infty}},$$
where $o(x)$ is a function which is smooth in a neighborhood of $x=0$ and satisfies that $o(0)=0$. In summary,this means that
$$\alpha_{0,-(m+1)}=\displaystyle \frac{\delta_0(c) + o(1-c)}{(1-c)^{m-d_0}}=\frac{\delta_\infty(c) + o(1+c)}{(1+c)^{m-d_\infty}},$$
where $$\lim_{c\rightarrow 1}\delta_0(c)>0 \quad \quad \lim_{c\rightarrow  -1}\delta_\infty(c)>0.$$
Similarly, we can show, albeit a bit more messily, that $\beta_{0,-(m+1)}$ is a rational function of $c$ which is
non-negative, has no vertical asymptotes within the interval $(-1,1)$ and satisfies 
$$\beta_{0,-m}=\displaystyle \frac{\gamma_0(c) + o(1-c)}{(1-c)^{m-d_0}}=\frac{\gamma_\infty(c) + o(1+c)}{(1+c)^{m-d_\infty}},$$
where $$\lim_{c\rightarrow 1}\gamma_0(c)>0 \quad \quad \lim_{c\rightarrow  -1}\gamma_\infty(c)>0.$$ We refer to proof of the generalized statement in Lemma \ref{betatrend} for more detail.
Putting this together we have 
$$H_S(c)= \frac{\left(\gamma_0(c) + o(1-c)\right)^{m+1}}{\left(\delta_0(c) + o(1-c)\right)^m} (1-c)^{-(m-d_0)}
= \frac{\left(\gamma_\infty(c) + o(1+c)\right)^{m+1}}{\left(\delta_\infty(c) + o(1+c)\right)^m} (1+c)^{-(m-d_\infty)},
$$
and hence clearly $\displaystyle \lim_{c\rightarrow \pm 1^{\mp}}H_S(c) = +\infty$. This together with the fact that $H_S(c)$ is smooth and bounded from below over the interval $(-1,1)$ allows us to conclude that $H_S(c)$ has at least one critical point $c\in (-1,1)$. For this value of $c$, we have $H_S'(c)=0$ and therefore $c$ solves \eqref{CSCobstruction} making $K_c$ a CSC Sasaki Reeb vector field. If $c$ is not a rational number, the corresponding CSC Sasaki structure will be irregular.
\end{proof}

\begin{remark}\label{cscsubcone}
It follows from the proof of Theorem \ref{CSCexistence} that the CSC Sasaki metric lies in the 2-dimensional subcone of $\gt^+(\cals)$ spanned by the constants plus $c\gz+1$. Note that here we view the Sasaki-Reeb cone in terms of Killing potentials.
\end{remark}

\begin{example}\label{Bottmanex}
As a simple example we let $N=\bbc\bbp^1\times \bbc\bbp^1$ and $N_\bfk$ be the total space of the twist 1 stage 3 Bott tower $\bbp(\BOne\oplus \calo(k_1,k_2))\lra \bbc\bbp^1\times \bbc\bbp^1$ as studied in \cite{BoCaTo17}. If $k_1k_2>0$ the K\"ahler manifold $N_\bfk$ does not have a CSCK metric. Nevertheless, Theorem \ref{CSCexistence} says that there is a CSCS ray in the Sasaki cone of the Boothby-Wang constructed Sasaki 7-manifold $M^7$ over $N_\bfk$. This also follows from the Main Theorem in \cite{BoTo21}. The CSCK obstruction here is the well known non-reductivity obstruction of Matsushima and Lichenerowicz. The topology of these 7-manifolds $M^7$ was described in \cite{BoTo21} where they were shown to have the rational cohomology of the 2-fold connected sum $2\#(S^2\times S^5)$ with torsion in $H^4$ which generally depends on both the pair $(k_1,k_2)$ and the integral K\"ahler class on $N_\bfk$. We refer to Section 3 of \cite{BoTo21} for details.
\end{example}

\begin{example}\label{K3ex}
Let $N$ be a projective K3 surface. Then $N$ has a non-trivial holomorphic line bundle $L$, so the $\bbc\bbp^1$ bundle $N^{ad}=\bbp(\BOne\oplus L)$ over $N$ is also non-trivial. Moreover, Theorem 7 of \cite{ACGT08} implies that the projective bundle $N^{ad}$ does not admit a CSC K\"ahler metric. But again, Theorem \ref{CSCexistence} says that there is a CSCS ray in the Sasaki cone of the Boothby-Wang constructed Sasaki 7-manifold $M^7$ over $N^{ad}$. Thus, choosing an admissible (up to scale)  primitive K\"ahler class on $N^{ad}$ gives a simply connected $M^7$ with $H^2(M^7,\bbz)\approx \bbz^{22}$ with a CSC Sasakian structure which may be irregular. Such 7-manifolds can be constructed as nontrivial $S^3$ bundles over K3 surfaces by using the fiber join construction of Yamazaki \cite{Yam99,BoTo20b}. Let us describe a bit more of the geometry and topology of such $M^7$. This gives manifolds of the form
$$S^3\lra M^7~\fract{\pi}{\longrightarrow}~ N, \qquad M^7=M^5_1\star_fM^5_2$$
where $M^5_i$ are Sasaki 5-manifolds formed from the Boothby-Wang construction over a projective K3 surface $N$ with integral K\"ahler form $\gro_i$. By Theorem 10.3.8 of \cite{BG05} $M^5_1$ and $M^5_2$ are both diffeomorphic to 
the connected sum $21\#(S^2\times S^3)$. Moreover, as described in \cite{BoTo20b} the 7-manifolds $M^7$ divide into two types depending on whether the classes $[\gro_1],[\gro_2]$ are colinear or non-colinear. From Proposition 3.12 of \cite{BoTo20b} and the fact that K3 has trivial canonical bundle, we see that the first Chern class of the contact bundle $\cald$ on $M^7$ satisfies $c_1(\cald)=-\pi^*([\gro_1]+[\gro_2])$. In the colinear case we have $\gro_i=l_i\gro$ where $\gro$ is primitive K\"ahler form. In this case we see that the second Stiefel-Whitney class $w_2(M^7)$ vanishes if and only if $l_1+l_2$ is even; hence, $M^7$ is spin when $l_1+l_2$ is even and non-spin when $l_1+l_2$ is odd. In both cases the 4-skeleton is formal, so for each such 3-sphere bundle it follows from Theorem 2.2 of \cite{KrTr91} that there is a finite number of diffeomorphism types. Thus, since $c_1$ is a contact invariant, in both cases there is a diffeomorphism class on $M^7$ that has a countable infinity of inequivalent contact structures of Sasaki type.

\end{example}

Combining Theorem \ref{CSCexistence} with Corollary 1.1 of \cite{CoSz12}, Corollary 1 \cite{ApJuLa21} or Theorem 2 \cite{ApCaLe21} gives the following statement concerning K-stability (on smooth equivariant test configurations).

\begin{corollary}\label{Kstabcor}
Under the hypothesis of Theorem \ref{CSCexistence} there exists a K-stable ray of Sasakian structures in the Sasaki-Reeb cone.
\end{corollary}

\begin{remark}
Note that Theorem 1 of \cite{ApCa18} (the proof of which is local in nature) also holds for Sasaki orbifolds. Further, for $d_0=d_\infty=0$ (the so-called {\em no-blow-down case}), replacing each ``$p_c(-1)$'' by ``$p_c(-1)/m_\infty$'' and each
``$p_c(1)$'' by ``$p_c(-1)/m_0$,''  where $m_0,m_\infty \in {\mathbb Z}^+$, in the proof of Theorem \ref{CSCexistence}, corresponds to generalizing the endpoint condition (iii) in \eqref{positivity} from
$$\Theta'(-1)=2\quad \text{and}\quad \Theta'(1)=-2.$$
to $$\Theta'(-1)=2/m_\infty\quad \text{and}\quad \Theta'(1)=-2/m_0.$$
Since this generalization does not affect the core of the proof, we can see that 
Theorem \ref{CSCexistence} generalizes within the no-blow-down case to the case of admissible metrics with orbifold singularities along the zero ($D_0$) and infinity ($D_\infty$) sections of  $N^{ad}=\bbp(E_0 \oplus E_{\infty})\to N$. In other words $N^{ad}$ can be replaced by the log pair $(N^{ad},\Delta)$, where
$\Delta = (1-1/m_0)D_0 + (1-1/m_\infty)D_\infty$. 
\end{remark}

Suppose now that the admissible manifold in Theorem \ref{CSCexistence} is Fano, i.e. $c_1(N^{ad})>0$. Consider a Boothby-Wang constructed  Sasaki manifold over $N^{ad}$ given by some K\"ahler form representing $c_1(N^{ad})/\cali$, where $\cali$  denotes the index of $N^{ad}$. Then $c_1(\cald)=0$ and so the CSC Sasaki metrics from Theorem \ref{CSCexistence} above are now $\eta$-Einstein. Second, we see from Lemma 5.2/Proposition 5.3 in \cite{BHL17}) that
since the basic first Chern class of the initial Sasaki metric is positive (as a pullback of the positive class $c_1(N^{ad})/\cali$), the average transverse scalar curvature of any Sasaki structure in the Sasaki-Reeb cone must be positive. In particular, the transverse (constant) scalar curvature of any $\eta$-Einstein structure in the cone must be positive. This means that any $\eta$-Einstein ray admits a Sasaki-Einstein structure. Thus, in this case
the corresponding Sasaki-Reeb cone will always have a (possibly irregular) Sasaki-Einstein structure (up to isotopy). Therefore, in the Gorenstein case we have the following corollary which in turn is a special case of the existence results in \cite{MaNa13} as well as Theorem 3 of \cite{ApJuLa21}.

\begin{corollary}\label{SEexistence}
Suppose $N^{ad}$ is a Fano admissible manifold $N^{ad}={P}(E_0 \oplus E_{\infty})\lra N$, where $N$ is a compact K\"ahler manifold which is a product of positive 
KE metrics.
Let $(M,\cals)$ be the Boothby-Wang constructed  Sasaki manifold given by $c_1(N^{ad})/\cali$. Then the corresponding Sasaki-Reeb cone will always have a
(possibly irregular) Sasaki-Einstein metric (up to isotopy).
\end{corollary}

\begin{example}\label{V1blowdownexample}
For a simple non-toric example using the above corollary, let us assume that $\cala = \{1\}$ and $N=N_1$ is the smooth Fermat sextic del Pezzo threefold $V_1=z_0^6+z_1^6+z_2^6+z_3^3+z_4^2=0$ in $\bbc\bbp(\bfw)=\bbc\bbp(1,1,1,2,3)$ described in Example 3.4.2. of \cite{Fanos} and \cite{ChSh09}. We have $|\bfw|-d=2$ so this is Fano with index 2. It is K-stable and admits a KE metric. Furthermore, since there is only one quadratic term in the polynomial describing $V_1$, its automorphism group is finite (see e.g. Proposition 37 in \cite{BGK05}) implying that it is a non-toric KE manifold. We also assume that $g_1$ is the KE metric on $V_1$ with $s_1=2$ and that
$d_0=0$. Then, independent of the value of $d_\infty\in \bbn \cup \{0\}$, the projective bundle $N^{ad}=\bbp(\BOne\oplus E_\infty)\lra N$ is Fano\footnote{If we instead assumed that $-g_1$ was KE and $s_1=-2$, then the Fano condition would be that $d_\infty=0$ whereas $d_0$ could be any value in $\bbn \cup \{0\}$.} and the K\"ahler class $c_1(N^{ad})/\cali$ is a positive scalar multiple of the admissible class determined by $x_1=\frac{d_\infty+2}{d_\infty+4}$ (see equations (3.8) and (3.11) in \cite{ACGT08b}). One can now calculate directly that for this value of $x_1$ and $c=0$,
 the left hand side of equation \eqref{CSCobstruction} equals $\frac{-3 \cdot 2^{2 d_\infty+11} (2 d_\infty+5) (11 d_\infty+28)}{(d_\infty+1) (d_\infty+4)^8 (d_\infty+5)}\neq 0$.
This means that the standard K\"ahler Futaki invariant does not vanish for $c_1(N^{ad})/\cali$, but Corollary \ref{SEexistence} gives us a quasiregular or irregular SE metric in the Sasaki-Reeb cone of the $(2(4+d_\infty)+1)$-dimensional, non-toric, Boothby-Wang constructed Sasaki manifold $(M,\cals_1)$ given by $c_1(N^{ad})/\cali$.
\end{example} 

\begin{remark}
In general, the Fano conditions for an admissible manifold $N^{ad}=\bbp(E_0 \oplus E_{\infty})\to N$ are summarized in Theorem 3.1 of \cite{ACGT08b}. If we assume no-blow-down here we have $d_0=d_\infty=0$. Specifically, assume $N=\prod_{a\in \cala} N_a$ is a finite product of compact positive K\"ahler-Einstein manifolds $(N_a,\pm g_a, \pm \omega_a)$ with scalar curvature $\pm 2 d_a s_a$ and let $N^{ad}=\bbp(\BOne \oplus \call)\lra N$, where $\call = \otimes_{a\in \cala} \call_a$ and $\call_a$ is (the pull-back of) a holomorphic line bundle over $N_a$ such that $c_1(\call_a) = [\frac{\omega_a}{2\pi}]$. Then the Fano condition is exactly the condition that $s_a>1$ if $\omega_a$ is positive and $s_a<-1$ if $\omega_a$ is negative. Since the Ricci forms $\rho_a=s_a \omega_a$, we see that a necessary condition for this to be a possible choice is that the index $\cali_a$ of each $N_a$ is at least $2$.
Further, from the details of Section 3 in \cite{ACGT08b} (specifically equation (3.8) in \cite{ACGT08b}) we observe that if $s_{a} \neq s_{\tilde{a}}$, then $x_{a} \neq x_{\tilde{a}}$, where the choice of $x_a$ for each $a\in\cala$ is governed by the K\"ahler class $c_1(N^{ad})/\cali$ (which is then an appropriate rescale of an admissible K\"ahler class).
Indeed, it is not hard to see that (3.8) in \cite{ACGT08b} for $d_0=d_\infty=0$ simply gives us $x_a=1/s_a$.
\end{remark}

\begin{remark}\label{bouqrem}
Varying the complex structure on $N$ gives rise to a bouquet of Sasaki-Reeb cones on $M$ as described in \cite{Boy10a,BoTo11,BoTo13}.
\end{remark}

\begin{example}\label{V1example}
Let us exhibit another example of a non-toric $M$ satisfying the assumptions in Corollary \ref{SEexistence} so that $c_1(N^{ad})/\cali$ itself is not KE.
For simplicity we will assume we are in the no-blow-down case and that $N=N_1\times N_2$ (so here $\hat\cala=\cala = \{1,2\}$).
Let $N_1$ be equal to $\bbc\bbp^2$ and pick $s_1=3/2$ (meaning $\omega_1$ is $4\pi$ times the unit volume Fubini-Study K\"ahler form). Then $x_1=2/3$.
Let $N_2$ be equal to the smooth Fermat sextic del Pezzo threefold $V_1$ described in Example \ref{V1blowdownexample}. So $N_2$ is Fano with index $2$, admits a KE metric, and has finite automorphism group. Therefore $N_1\times N_2$ is a non-toric K\"ahler manifold.

We now consider the projective bundle $N^{ad}$ over $N_1\times N_2$. Since the Fano index is $2$, we can pick $s_2=2$. Then $x_2=1/2$. 
For the admissible data $(s_1,s_2,x_1,x_2,d_1,d_2)=(3/2, 2,2/3,1/2,2,3)$ and $c=0$ we can easily check that \eqref{CSCobstruction} does not hold. Equivalently,
plugging this admissible data into (7) of \cite{ACGT08}, we calculate 
that $\alpha_1 \beta_0-\alpha_0\beta_1\neq 0$. According to Proposition 6 in \cite{ACGT08} we therefore have that the Futaki invariant does not vanish for  $c_1(N^{ad})/\cali$. 
This means that $c_1(N^{ad})/\cali$ itself is not KE. Corollary \ref{SEexistence} now gives us a quasiregular or irregular SE metric in the Sasaki-Reeb cone of the $13$-dimensional, non-toric, Boothby-Wang constructed Sasaki manifold $(M,\cals_1)$ given by $c_1(N^{ad})/\cali$. Note that since $x_1\neq x_2$, this example does not overlap with any of the examples arising from $S_\bfw^3$-joins (see \cite{BoTo14a,BoTo19a}). As described in item 15 of \cite{BGK05} (since $V_1$ has no continuous automorphisms), the dimension of the local moduli space is
$$h^0(\bbc\bbp(\bfw),\calo(d))-\sum_ih^0(\bbc\bbp(\bfw),\calo(w_i))=$$
$$h^0(\bbc\bbp(\bfw),\calo(6))-h^0(\bbc\bbp(\bfw),\calo(1))-h^0(\bbc\bbp(\bfw),\calo(2))-h^0(\bbc\bbp(\bfw),\calo(3)).$$
We get $h^0(\bbc\bbp(\bfw),\calo(6))=59,~~h^0(\bbc\bbp(\bfw),\calo(1))=3,~~h^0(\bbc\bbp(\bfw),\calo(2))=7,$ and $h^0(\bbc\bbp(\bfw),\calo(3))=13$ giving a 36-dimensional complex family. This gives rise to a bouquet of one dimensional Sasaki-Reeb cones on the $S^1$ bundle over $V_1$ corresponding to the primitive K\"ahler class which in turn induces a bouquet of Sasaki-Reeb cones on the 13-dimensional manifold $(M,\cals_1)$.
\end{example}

It should be clear that many such examples can be constructed using results in \cite{BGK05} as well as \cite{Fanos}.
So we now consider a fairly large class $\calc$ of simply connected algebraic varieties generalizing our example. We consider Brieskorn-Pham polynomials of degree $d$ with exponent vector $\bfa\in (\bbz^+)^{n+1}$, weight vector $\bfw\in (\bbz^+)^{n+1}$ of the form
\begin{equation}\label{whp}
f_\bfa(\bfz)= z_0^{a_0} +\cdots +z_n^{a_n}, \qquad a_i=\frac{d}{w_i}.
\end{equation}
We do not want linear terms in $f_\bfa$ since they give rise to the standard sphere. So we shall assume that all $a_j\geq 2$. There is a natural weighted $\bbc^*$ action on $\bbc^{n+1}$ defined by
\begin{equation}\label{C*act}
\grl\cdot \bfz=(\grl^{w_0}z_0,\ldots,\grl^{w_n}z_n)
\end{equation}
which induces 
$$f_\bfa(\grl\cdot\bfz)=\grl^df_\bfa(\bfz).$$
So the Brieskorn manifold $M_\bfa$ defined by $M_\bfa=(f_\bfa(\bfz)=0)\cap S^{2n+1}$ has a natural weighted $S^1$ action defined by restriction. The weighted $\bbc^*$ quotient of the hypersurface $\bigl(f_\bfa(\bfz)=0\bigr)\setminus \{0\}$ or equivalently the weighed $S^1$ quotient of $M_\bfa$ gives weighted projective varieties of the form
\begin{equation}\label{proj}
V_\bfa={\rm Proj}~\bbc[z_0,\ldots,z_n]/\bigl(f_\bfa\bigr)
\end{equation} 
where $(f)$ denotes the ideal generated by $f$. These varieties all embed in the weighted projective space $\bbc\bbp(\bfw)$. 
Brieskorn manifolds $M_\bfa$ also have a natural Sasakian structure whose Reeb vector field $\xi$ generate the $S^1$ action. In order to satisfy the hypothesis of Theorem \ref{CSCexistence} we want the $S^1$ quotient of $M_\bfa$ to be a smooth manifold with a trivial orbifold structure\footnote{There are such smooth projective varieties with non-trivial orbifold structures, but we do not consider them here.}. Equivalently we want the Sasakian structure on $M_\bfa$ to be regular (Proposition 9.3.22 of \cite{BG05}):

\begin{proposition}\label{regBries}
A Brieskorn manifold has a regular Sasakian structure $\cals_\bfw$ if and only if the weights are pairwise relatively prime.
\end{proposition}

We are interested in regular Sasaki-Brieskorn manifolds whose projective algebraic quotient is Fano. The latter condition means that the Sasakian structure is positive which for a Brieskorn manifold is the condition $|\bfw|-d>0$ where $|\bfw|=\sum_jw_j$. Here we restrict ourselves to dimensions 5 and 7 which correspond to $n=3,4$. As previously, we assume that $a_j\geq 2$, so standard spheres are eliminated. 

\begin{proposition}\label{reglow}
Let $M_\bfa$ be a regular Brieskorn manifold of dimension 5 or 7 with a positive Sasakian structure. Then 
\begin{enumerate}
\item if $n=3$; we have $\bfa=(2,2,2,2),(3,3,3,3),(4,4,4,2),(6,6,3,2)$;
\item if $n=4$; we have $\bfa=$
$$\qquad \qquad (2,2,2,2,2),~ (3,3,3,3,3),~ (4,4,4,4,4),~ (4,4,4,4,2),~(6,6,6,6,2) ,~(6,6,6,3,2).$$
\end{enumerate}
\end{proposition}

\begin{proof}
The conditions that we must satisfy are $|\bfw|-d>0$, the weights $w_j$ are pairwise relatively prime, and divide $d$ with no $w_j$ equal to $d$.
We do the case $n=4$. First $\bfw=(1,1,1,1,1)$ implies $d=2,3,4$ giving the first three on the list. If $\bfw=(1,1,1,1,b)$ then we must have $b+4>d=ab$ for some $a\geq 2$. This implies that $a=2$ and $b=2,3$ giving the next two. Next if $\bfw=(1,1,1,b,c)$ with $b<c$ relatively prime, then we have $b+c+3>ac$ for some $a\geq 2$. So $b+3>c$ which implies that $c=b+1$ or $b+2$. This implies $\bfw=(1,1,1,b,b+1)$ or $\bfw=(1,1,1,b,b+2)$. The only solution that satisfies $|\bfw|>d$ is $\bfw=(1,1,1,2,3)$ with $d=6$. One can check that there are no further solutions.
\end{proof}

In particular, as long as the Fano index of the projective variety $V_\bfa$ of Equation \eqref{proj}, with $\bfa$ given in Proposition \ref{reglow}, is at least $d_0+2$, we can take $N$ in Corollary \ref{SEexistence} to be $V_\bfa$. 

\subsection{The Sasaki-Futaki invariant in the admissible case}\label{admissibleSFInvariant}
Assume that we are in the admissible case as briefly described in Section 2.3 and assume that the admissible K\"ahler class $\Omega$ is rational
and is represented by an admissible metric $g$ corresponding to a smooth function $F(\gz)=\Theta(\gz)p_c(\gz)$ such that
\eqref{positivity} is satisfied.  After an appropriate (and otherwise ignored) rescaling, this defines a Boothby-Wang constructed Sasaki manifold $(M,\cals)$ with Reeb vector field $X$. 

Let $K_c$ denote the Reeb vector field in the Sasaki-Reeb cone of $(M,\cals)$ induced by $f_{K_c}=c\,\gz+1$.
That is, $\eta^X_\cald (K_c)=c\,\gz+b$ (when $\gz$ is lifted to $(M,\cals)$). It now follows from Section 2 of \cite{ApMaTF18} that the weighted scalar curvature $Scal_{c\,z+1,m+2}(g)$ is given by
$$Scal_{c\,\gz+1,m+2}(g)=\frac{-(c\,\gz+1)^{m+3}G''(\gz)}{p_c(\gz)}+2(c\,\gz+1)^2\Big(\sum_{a\in \hat{\cala}} \frac{x_ad_as_a}{1+x_a \gz}\Big),$$
where $G(\gz):= \frac{F(\gz)}{(c\,\gz+1)^{m+1}}$ and $m=(\Sigma_{a\in\hat\cala}d_a)+1$.

Now, consider the Futaki Invariant $Fut_{K_c}(Z)$ of \eqref{Futinv}
for any $Z \in \gt$. Using Remark B.2. of \cite{ApCaLe21}, we have that, up to a multiple by a non-zero constant, 
$$Fut_{K_c}(Z) = \int_{-1}^1\left(Scal_{c\,\gz+1,m+2}(g)-C_{K_c}\, (c\,\gz+1)\right)f_Z (c\,\gz+1)^{-m-3}p_c(\gz)\,d\gz,$$
where $f_Z=\eta^X_\cald(Z)$ and, using (48) of \cite{ApCaLe21},
$$C_{K_c} = \frac{ \int_{-1}^1 Scal_{c\,\gz+1,m+2}(g)(c\,\gz+1)^{-m-2} p_c(\gz)\,d\gz}{ \int_{-1}^1(c\,\gz+1)^{-m-1} p_c(\gz)\,d\gz} = \frac{2\beta_{0,-m}}{\alpha_{0,-m-1}}.$$

Let $Z= \gz X$ with $\gz$ lifted to $(M,\cals)$. We can view $Z$ as $\frac{d}{dc} K_c$.
Using the useful observation from the proof of Theorem \ref{CSCexistence} that
\begin{itemize}
\item $c\, \alpha_{r+1,k}+\alpha_{r,k}=\alpha_{r,k+1}$
\item $c\, \beta_{r+1,k}+\beta_{r,k}=\beta_{r,k+1}$,
\end{itemize}
it is straightforward to verify that 
\begin{eqnarray}
Fut_{K_c}(Z) & = &\frac{2(\alpha_{0,-(m+1)}\beta_{1,-(m+1)}-\alpha_{1,-(m+2)}\beta_{0,-m})}{\alpha_{0,-(m+1)}}\\
&=& \frac{2(\alpha_{1,-(m+2)}\beta_{0,-(m+1)}-\alpha_{0,-(m+2)}\beta_{1,-(m+1)})}{\alpha_{0,-(m+1)}} \notag \\
&=&  \frac{2}{\alpha_{0,-(m+1)}} \bigr(\alpha_{1,-(m+2)}\beta_{0,-(m+1)}-\alpha_{0,-(m+2)}\beta_{1,-(m+1)}\bigl)  \notag.
\end{eqnarray}
Thus, equation \eqref{CSCobstruction} is equivalent to the vanishing of $Fut_{K_c}(Z)$.

In conclusion, \eqref{CSCobstruction} is an obstruction to the existence (up to isotopy) of a CSC Sasaki structure in the ray determined by $K_c$. This also follows from
the more ``expensive'' Theorem 3 in \cite{ApCaLe21} and the discussion leading up to \eqref{CSCobstruction}.
Indeed, combining Theorem 3 in \cite{ApCaLe21} and Proposition 2.2 of \cite{ApMaTF18} with the above observations we can state a very explicit existence criterion.
To do so we need a few preliminary notations lifted straight from \cite{ApMaTF18} with some adaptation for our choice of $f=c\, \gz+1$:

Let $A_1$ and $A_2$ be given by the unique solution to the linear system
\eqref{A1andA2} and let 
$$Q(\gz) =\frac{p_c(\gz)}{(c\,\gz+1)^{m+1}} \Big(\sum_{a\in \hat{\cala}} \frac{2x_ad_as_a}{1+x_a \gz}\Big)-\frac{(A_1\gz + A_2)p_c(\gz)}{(c\,\gz+1)^{m+3}}.$$
Now let
$$G(\gz) = \frac{2p_c(-1)}{(1-c)^{m+1}}(\gz+1) + \int_{-1}^\gz Q(t)(\gz-t)\,dt$$
and consider the {\em weighted extremal polynomial}
\begin{equation}\label{wextrpol}
F_{\Omega,c,m+2}(\gz)=(c\,\gz+1)^{m+1}G(\gz).
\end{equation}

\begin{proposition}\label{existencecriterion}
Let $\Omega$ be a rational admissible K\"ahler class on an admissible K\"ahler manifold $N^{ad}=\bbp(E_0\oplus E_\infty) \rightarrow N$
 where $N$ is a compact K\"ahler manifold which is a local product\footnote{See Definition \ref{adkahman} for what we mean by ``local product'' here.} of CSCK metrics. Pick any admissible K\"ahler metric $g$ with K\"ahler form in $\Omega$ and let the moment map be given by $\gz: N^{ad}\rightarrow [-1,1]$. Let $(M,\cals)$ be the Boothby-Wang constructed Sasaki manifold given by an appropriate rescale of $\Omega$. Let $c\in (-1,1)$ and consider the corresponding weighted extremal polynomial $F_{\Omega,c,m+2}(\gz)$ given by \eqref{wextrpol}.
 Then 
 \begin{itemize}
\item \cite{ApCaLe21} the Reeb vector field $K_c$ determined by $f_{K_c}=c\,\gz+1$ is extremal (up to isotopy) if and only if
 $$F_{\Omega,c,m+2}(\gz)>0,\quad -1<\gz<1$$
 \item the Reeb vector field $K_c$ determined by $f_{K_c}=c\,\gz+1$ is CSC (up to isotopy) if and only if
 $$F_{\Omega,c,m+2}(\gz)>0,\quad -1<\gz<1, \quad \text{and} \quad \underbrace{\alpha_{1,-m-2}\beta_{0,(-m-1)}-\alpha_{0,-m-2}\beta_{1,(-m-1)}=0}_{Equation \eqref{CSCobstruction}},$$
 where $\alpha_{r,k}$ and $\beta_{r,k}$ are given by \eqref{alphabeta}.
 \end{itemize}
 \end{proposition}
 
 \begin{remark}\label{extrpol}
 Note that $F_{\Omega,0,m+2}(\gz)$ is the usual admissible extremal polynomial for $\Omega$ (see Definition 1 in \cite{ACGT08}).
 \end{remark}

\subsection{Sasaki-Reeb cones with extremal but no CSC Sasaki metrics}
Existence results like Theorem \ref{CSCexistence} give rise to the question of whether a Sasaki CR structure with extremal metrics always has a CSC Sasaki metric in its Sasaki-Reeb cone. This was formulated as Problem 6.1.2 in \cite{BHLT21}. Here we give a negative answer to this question, namely:

\begin{theorem}\label{nocscthm}
There exist Sasaki CR structures with extremal Sasaki metrics, but no CSC Sasaki metrics.
\end{theorem}

\begin{proof}
We use the fact that a Sasakian structure $(M,\cals)$ in the extremal subcone $\ge(\cald,J)\subset \gt^+(\cald,J)$ is CSC if and only if its Futaki invariant $\calf_\cals$ vanishes and we note that $\calf_\cals$ only depends on the isotopy class \cite{BGS06}. The proof proceeds by constructing an appropriate Sasaki CR structure over a certain smooth projective algebraic variety as described by the following example.
\end{proof}

\begin{remark}\label{dim1rem}
Before presenting the example we note that Theorem \ref{nocscthm} is false when restricted to Sasaki CR structures with a 1-dimensional Sasaki-Reeb cone.
\end{remark}

\begin{example}\label{prob612}
We begin by revisiting the example presented carefully in Section 6 of \cite{MaTF11}. The admissible projective variety in question is $N^{ad} = \bbp(\calo \oplus \calo(1,-1)) \to \Sigma_1 \times \Sigma_2$, where $\Sigma_i$ are compact genus two Riemann surfaces and $g_1$ and $-g_2$ are K\"ahler metrics on $\Sigma_1$ and $\Sigma_2$, respectively that have constant scalar curvature equal to $-4$. Following the notation in \cite{ACGT08}, this means that $s_1=-2$, $s_2=2$, $0<x_1<1$, and $-1<x_2<0$, where the choice of $\bfx=(x_1,x_2)$ determines the K\"ahler class
$$ \Omega_{\bfx} = \frac{[\omega_1]}{x_1} + \frac{[\omega_2]}{x_2} + \Xi,$$
where $\Xi$ is the Poincare dual of $2\pi(e_0+e_\infty)$. Since here $[\omega_i]$ and $\Xi$ are primitive integer classes, we see that $\Omega_\bfx$ is $2\pi$ times a rational class
if and only if $x_i$ are rational numbers. Note that this example does not satisfy the conditions assumed in Theorem \ref{CSCexistence}.

Now, as explained in Section 6 of \cite{MaTF11}, when $x_1=x$ and $x_2=-x$ for any $0<x<1$, the Futaki invariant of the corresponding K\"ahler class vanishes (so any extremal metric in that class would be CSC) and the class contains a CSC K\"ahler metric if and only if the {\em extremal polynomial}, $F_x(\gz) = \frac{(1-\gz^2)(6 - 7x^2 -4x^3 + x^4 - x^2 (1 - 4x -x^2)\gz^2)}{2(3-x^2)}$, is positive for $-1<\gz<1$. The latter follows from Theorem 1 in \cite{ACGT08}. As also explained in \cite{MaTF11}, for any $0<x<1$ sufficiently small, this positivity holds and we have CSC K\"ahler metrics in the corresponding class (as expected from Theorem 1 of \cite{ACGT08}), but when e.g. $x=4/5$, positivity fails and the corresponding class has no CSC K\"ahler metric and no extremal K\"ahler metric at all. The class can still be represented by some admissible K\"ahler metric. Moving forward we assume $x_1=x$, $x_2=-x$, and $x=4/5$. Note that $F_{4/5}(\gz) =\frac{ (1-\gz^2)(568 \gz^2-37)}{1475}$.

Consider a Boothby-Wang constructed Sasaki manifold $(M,\cals)$ above $N^{ad}$ with respect to an appropriate rescale of the K\"ahler class $\Omega_{4/5}$ determined by
$x_1=4/5$, $x_2=-4/5$ (pick e.g. a (rescale of an) admissible K\"ahler metric in $\Omega_{4/5}$).
Then the (rescales of) Reeb vector fields $K_c$ induced by the Killing potentials $c\,\gz+1$ for $c \in (-1,1)$ span the entire Sasaki-Reeb cone of $(M,\cals)$ (since genus of $\Sigma_i=2$, this cone has dimension $2$)). 

Using $\eqref{A1andA2}$ for this particular case (where $p_c(t)=(1-(4/5)^2 t^2)$) with $p=5$, we calculate that 
$A_1=-\frac{4 c \left(535 c^4+11566 c^2-17933\right)}{115 c^4-3686 c^2+4543}$ and  
$A_2=\frac{2 \left(12305 c^6-13795 c^4+6227 c^2-16401\right)}{115 c^4-3686 c^2+4543}$.
By Proposition \ref{existencecriterion}, we see that
a ray in the Sasaki-Reeb cone corresponding to a value $c \in (-1,1)$ admits an extremal Sasaki metric (up to isotopy) if and only if
\tiny
$$ F_{\Omega_{4/5},c,5}(\gz)=\frac{(1-\gz^2)  \left(115 c^4  \left(553 z^2-535\right)+110070 c^3 \gz(1-\gz^2)+c^2(98651-164999  \gz^2)-149760 \,c \,\gz (1-\gz^2)+154 \left(568 \gz^2-37\right)\right)}{50 \left(115 c^4-3686 c^2+4543\right)}$$
 \normalsize
is positive for $-1<\gz < 1$. 

As expected from Remark \ref{extrpol}, $F_{\Omega_{4/5},0,5}(\gz) =F_{4/5}(\gz)$. So for $-1<c<1$ sufficient close to zero, this positivity certainly fails. On the other hand, one can easily verify that
$$F_{\Omega_{4/5},-1,5}(\gz)=\frac{1}{300} (1-\gz^2)  \left(-245 \gz^3-86 \gz^2+245 \gz+194\right)$$ and $$F_{\Omega_{4/5},1,5}(\gz)=\frac{1}{300} (1-\gz^2)  \left(245 \gz^3-86 \gz^2-245 \gz+194\right)$$ both are positive for $-1<\gz < 1$. Thus we may conclude that for $c \in (-1,1)$ sufficiently close to either $-1$ or $+1$, $F_{\Omega_{4/5},c,5}(\gz)$ is positive for $-1<\gz<1$.

We also can calculate that in this case equation \eqref{CSCobstruction} is equivalent to
\begin{equation}\label{FutExNoCSC}
\frac{4 c  \left(-2461 c^4+512 c^2+3893\right)}{5625 (1-c^2)^7}=0.
\end{equation}
Note that this equation has no solution for $0<|c|<1$.

From the above work we see that, while $\cals$ itself (``initial ray'') is not extremal, for $c \in (-1,1)$ sufficiently close to either $-1$ or $+1$
(``far enough away from the initial ray''), the corresponding ray does have an extremal Sasaki metric (up to isotopy). 
Moreover, since here equation \eqref{CSCobstruction} (i.e. equation \eqref{FutExNoCSC}) has no solution for all $0<|c|<1$, there are no CSC rays in the Sasaki-Reeb cone at all. [Essentially, the only shot at a CSC ray would have been the initial ray of the Boothby-Wang construction, but since this ray
allows no extremal metric at all, it is not a CSC ray either.] This example proves Theorem \ref{nocscthm} and tells us that the answer is "no" to the question in Problem 6.1.2. of \cite{BHLT21}.

We can dig a little deeper and numerically investigate the extremal Sasaki cone $\ge(\cald,J)$ in $\gt^+(\cald,J)$. In this case we identify
 $\gt^+(\cald,J)$ with the interval $\calc := \{ c\,\,|\, -1<c<-1 \}$ swept out by transverse homotheties, and then correspondingly $\ge(\cald,J)$ is identified with 
 $$\calc_\ge := \{c\in \calc\,|\, \forall \gz \in (-1,1),\, F_{\Omega_{4/5},c,5}(\gz)>0\}=  \{c\in \calc\,|\, \forall \gz \in (-1,1),\, P(\gz,c)>0\},$$
 swept out by transverse homotheties where
 $$
 \begin{array}{ccl}
 P(\gz,c) & = & 115 c^4  \left(553 \gz^2-535\right)+110070 c^3 \gz(1-\gz^2)+c^2(98651-164999  \gz^2)\\
 \\
&- &149760 \,c \,\gz (1-\gz^2) + 154 \left(568 \gz^2-37\right)\\
\\
&=& -5698 + 98651 c^2 - 61525 c^4 - 149760 c \gz + 110070 c^3 \gz + 
 87472 \gz^2 \\
 \\
 &- & 164999 c^2 \gz^2 + 63595 c^4 \gz^2 + 149760 c \gz^3 - 
 110070 c^3 \gz^3.
 \end{array}
 $$
 Equivalently, $\calc\setminus \calc_\ge = \{ c\in \calc\,|\, \exists \gz \in (-1,1)\,s.t.\,P(\gz,c)\leq 0\}$. To understand the latter set we need to see what part of the graph of 
 $y=P(\gz,c)$, as defined over the square $\{(\gz, c)\,|\, -1<\gz<1,\, -1<c<1\}$, is below or at  ``sea level'' (i.e. at or below $y=0$). 
 
 Rather than attempting a rigorous treatment here, we shall rely on Mathematica to give us some clue: The plot of $y=P(\gz,c)$ looks as follows:
 
 \includegraphics[scale=0.6]{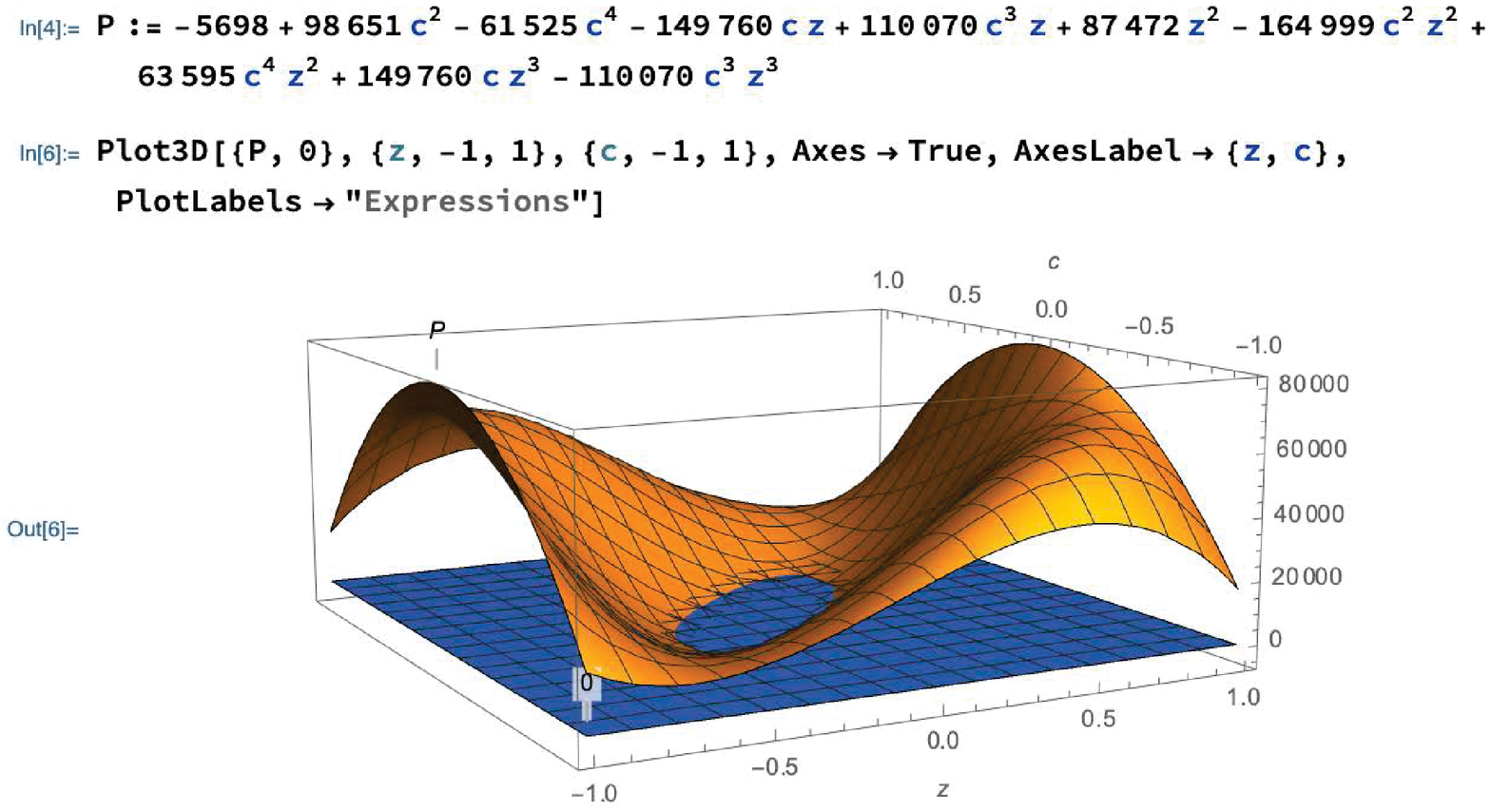}
 
 From this picture it looks like $\calc\setminus \calc_\ge$ is one connected subset of $\calc$, splitting $\calc_\ge$ into two connected pieces. This is further underscored by the fact that the plot of $P(\gz,c)=0$ in the square $\{(\gz, c)\,|\, -1<\gz<1,\, -1<c<1\}$ appears to be a simple closed curve:
  \includegraphics[scale=0.4]{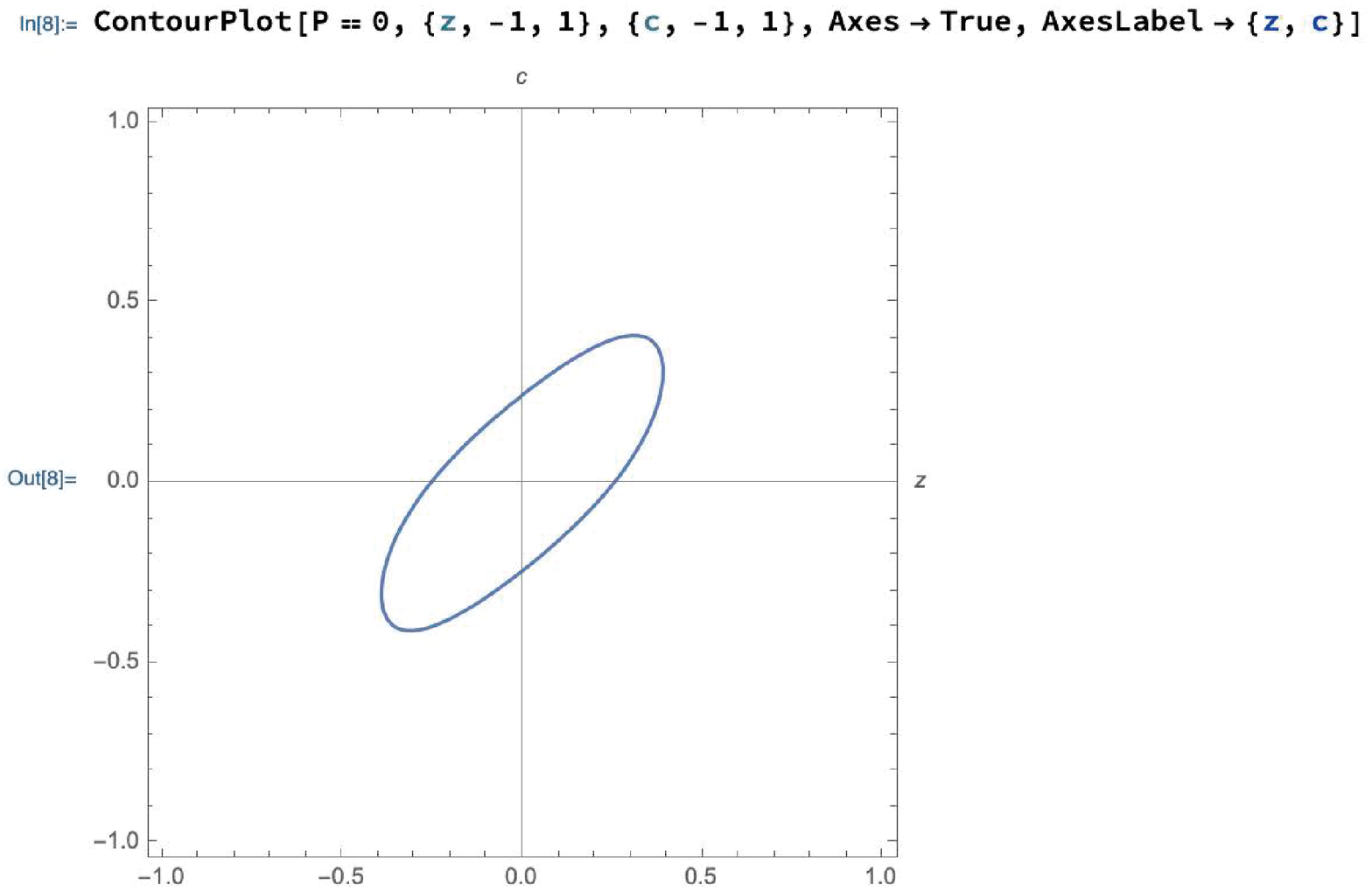}
 
 We are now interested in finding the maximum and minimum $c$-values on this curve. This should give an estimate of $\calc\setminus \calc_\ge$, which seems to be an interval.
 To that end, by the theory of Lagrange Multipliers together with the picture at hand, we need to solve the system 
 $$
 \begin{array}{ccl}
\frac{\partial P}{\partial \gz} & = &0\\
 \\
 P&=&0
 \end{array}
 $$
 over the interval $\{(\gz, c)\,|\, -1<\gz<1,\, -1<c<1\}$. We calculate that
 $$\frac{\partial P}{\partial \gz} =2 (-74880 c + 55035 c^3 + 87472 \gz - 164999 c^2 \gz + 63595 c^4 \gz + 
   224640 c \gz^2 - 165105 c^3 \gz^2)$$ and find that numerically the above system has two solutions
   in $\{(\gz, c)\,|\, -1<\gz<1,\, -1<c<1\}$ with $c$ values approximately equal to $\pm 0.410752$:
 
 \includegraphics[scale=0.6]{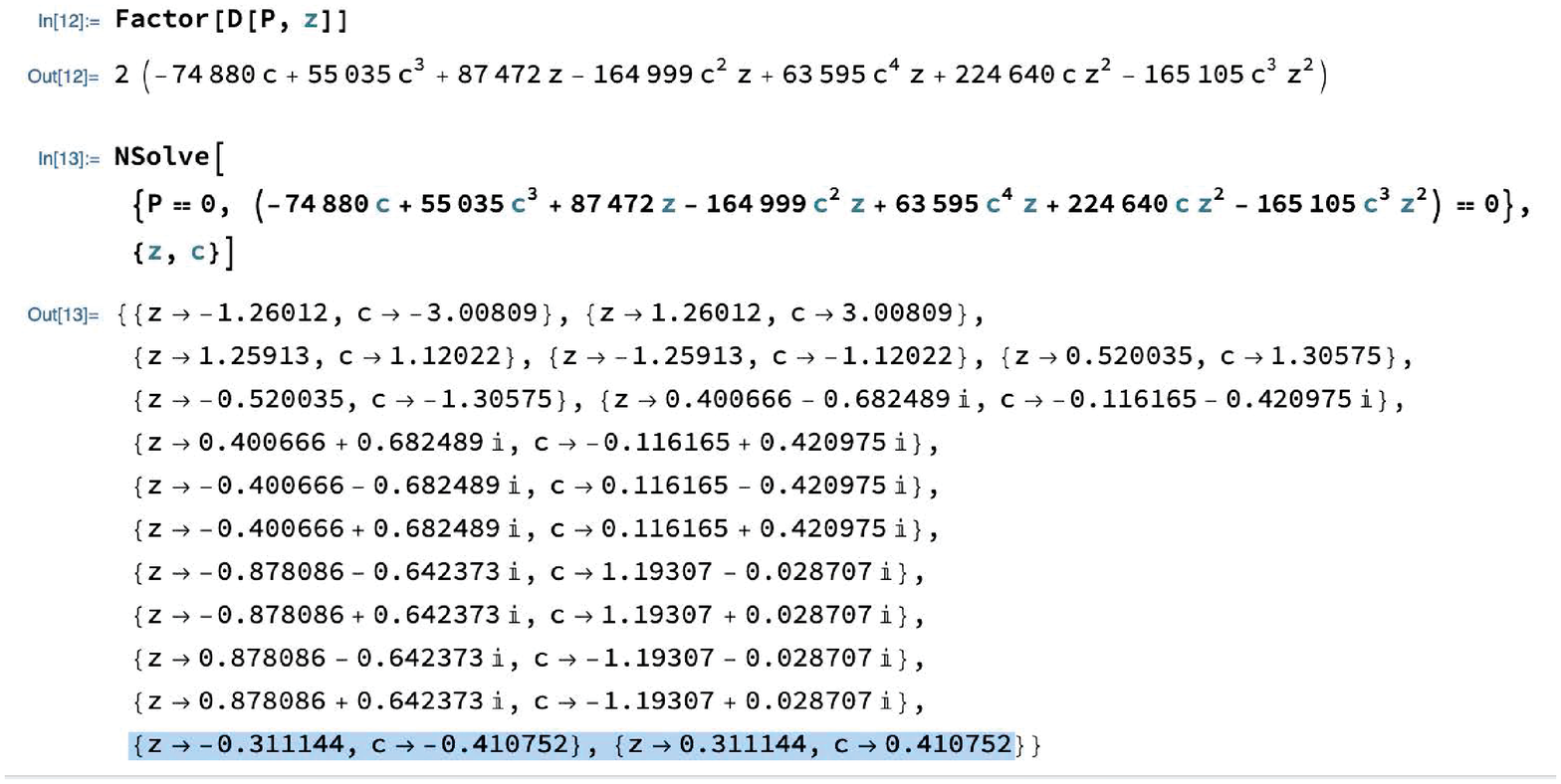}
In conclusion, {\em the numerical evidence gathered from Mathematica tells us that $\calc\setminus \calc_\ge = [-\hat{c},\hat{c}]$ and $\calc_\ge = (-1,-\hat{c}) \cup (\hat{c},1)$, where $\hat{c} \approx 0.410752$.}
\end{example}
 
 Another example can be given using the Yamazaki fiber join construction.

 \begin{example}
 Note that if instead we have $N^{ad} = \bbp(\calo \oplus \calo(8k_2^1,-8k_1^2)) \to \Sigma_1 \times \Sigma_2$, where genus of 
 $\Sigma_1$ is $1+8k_2^1$, and genus of 
 $\Sigma_2$ is $1+8k_1^2$, we still have $s_1=-2$ and $s_2=2$ as in Example \ref{prob612}. If we further use the K\"ahler class corresponding to $x_1=-x_2=4/5$ on $N^{ad}$, we have the same calculations and conclusion as above. In this case, we can view $N^{ad}$ with $\Omega$ as the regular quotient of a Yamazaki fiber join as described in Section 5.3 of
 \cite{BoTo20b} with $K=\begin{pmatrix} 9k_2^1&k_1^2\\
 k_2^1&9k_1^2 \end{pmatrix}.$
 \end{example}

\subsection{Sasaki-Reeb cones with no extremal Sasaki metrics} 
There are many examples of Gorenstein Sasaki-Reeb cones with no extremal metrics whatsoever \cite{BovC16}. Here we provide examples in the non-Gorenstein case, i.e. when the polarization is not that of the canonical or anticanonical bundle.

\begin{example}\label{nonexistenceex}
We will assume $N^{ad}= \bbp(\calo \oplus \calo(1,1)) \to \Sigma_1 \times \Sigma_2$, where $\Sigma_i$ are compact Riemann surfaces of
genus $101$ and $51$ respectively. Let $g_1$ and $g_2$ be K\"ahler metrics on $\Sigma_1$ and $\Sigma_2$, with constant scalar curvature equal to $-400$ and $-200$, respectively. This means that $s_1=-200$, $s_2=-100$, $0<x_i<1$. Again, using Proposition \ref{existencecriterion}, we calculate that for a Boothby-Wang  constructed Sasaki manifold with respect to an appropriate rescale of the rational K\"ahler class $\Omega$ determined by $x_1=100/101$ and $x_2=9/10$, respectively, the Reeb vector field $K_c$ induced by the Killing potential $c\,\gz+1$ for $c \in (-1,1)$ is extremal if an only if 
 $$F_{\Omega,c,5}(\gz)=\frac{(1-\gz^2)g_c(\gz)}{2020\left(9714895 c^4-41234400 c^3+65557582 c^2-46274160 c+12236095\right)}$$
 is positive for $-1<\gz<1$, where
 \tiny
 $$
 \begin{array}{ccl}
  g_c(\gz) & = & -2 \left(230646399215 c^4-1016343615871 c^3+1668631275785 c^2-1210756275690 c+327823027600\right)\\
  \\
  &+& (-338414479445 c^4+1516758787660 c^3-2526484662221 c^2+1856228939280 c-508090145290)\gz\\
  \\
  &+& 2 \left(249201848665 c^4-1095101319871 c^3+1793846257405 c^2-1299139921290 c+351193969050\right)\gz^2\\
  \\
  &+& (375505948555 c^4-1674191726860 c^3+2776783510297 c^2-2032903682160 c+554807556000) \gz^3.
  \end{array}
  $$
  \normalsize
 Since 
 $$
 \begin{array}{cl}
 &9714895 c^4-41234400 c^3+65557582 c^2-46274160 c+12236095\\
 \\
 =& 9714895 (1-c)^4+2374820 (1-c)^3+143752 (1-c)^2+2616 (1-c)+12\\
 \\
 >& 0
 \end{array}
 $$ 
 for all $c\in (-1,1)$, positivity of $F_{\Omega,c,5}(\gz)$ for $-1<\gz<1$ is equivalent to positivity of
 $g_c(\gz)$ for  $-1<\gz<1$.
 However, e.g. 
\tiny
$$g_c(0) = -2 \left(230646399215 c^4-1016343615871 c^3+1668631275785 c^2-1210756275690 c+327823027600\right)$$
\normalsize
is negative for all $c\in (-1,1)$, so
positivity of $F_{\Omega,c,5}(\gz)$ for  $-1<\gz<1$ will not hold for any value of $c\in (-1,1)$. Similarly to Example \ref{prob612}, the Sasaki-Reeb cone here is parametrized by $c\in (-1,1)$ and thus it follows that,
in this case, the Sasaki-Reeb cone has no extremal rays at all.
 \end{example}

\begin{example} 
 Note that if we instead we have $N^{ad} = \bbp(\calo \oplus \calo(200k_2^1,18k_2^2)) \to \Sigma_1 \times \Sigma_2$, where genus of 
 $\Sigma_1$ is $20000 k_2^1+1$, and genus of 
 $\Sigma_2$ is $900 k_2^2 +1$, we still have $s_1=-200$, $s_2=-100$ as in Example \ref{nonexistenceex}. If we further use the K\"ahler class corresponding to $x_1=100/101$ and $x_2=9/10$ on $N^{ad}$, we have the same calculations and conclusion as above. In this case, we can view $N^{ad}$ with $\Omega$ as the regular quotient of a Yamazaki fiber join as described in Section 5.3 of
 \cite{BoTo20b} with $K=\begin{pmatrix} 201k_2^1&19 k_2^2\\
 k_2^1&k_2^2 \end{pmatrix}$. Thus there is a countably infinite family of Yamazaki fiber joins that have no extremal Sasaki metrics at all. 

\end{example}

\begin{remark}
These non-Gorenstein examples all have a $2$-dimensional Sasaki-Reeb cone $\gt^+$, as opposed to the examples in \cite{BovC16}) where $\gt^+$ can be bigger.
\end{remark}

\section{The topology of $(M,\cals)$}
Here we describe, more generally, the topology of the Boothby-Wang Sasaki manifold over the projective bundle $N^{ad}$ some of which has been described in previous work \cite{BoTo21} as well as briefly discussed in examples \ref{Bottmanex} and \ref{K3ex}. 

\begin{proposition}\label{BWsphere}
Let $S^1\lra M\lra N^{ad}=\bbp(E_0\oplus E_\infty)$ be the Boothby-Wang-Sasaki bundle. Then $M$ is the total space of a lens space bundle over $N$.  
\end{proposition}

\begin{proof}
We have the commutative diagram of oriented fiber bundles 
\begin{equation}\label{cartfib}
\xymatrix{
                & & \bbc\bbp^{d_0+d_\infty+1} \ar[d]  \\
S^1 \ar[r] &M \ar[dr]\ar[r]_{\pi_M}  & N^{ad} \ar[d]^{\pi_N}  \\
&& N 
}
\end{equation}
We claim that the fibers of the southeast arrow are lens spaces. By definition the southeast arrow is the projection $\pi_N\circ \pi_M$ which is clearly smooth. Let $p\in N$, the fiber  $\pi_N^{-1}(p)$ is the projective space $\bbc\bbp^{d_0+d_\infty+1}$, and $\pi_M^{-1}\bigl(\pi_N^{-1}(p)\bigr)$ is an $S^1$ bundle over $\pi_N^{-1}(p)=\bbc\bbp^{d_0+d_\infty+1}$ which is classified by $H^2(\bbc\bbp^{d_0+d_\infty+1},\bbz)\approx \bbz$. Since $M$ has a Sasakian structure, the bundle is determined by some $l\in\bbz^+$. So the fiber of $\pi_N\circ\pi_M$ is a lens space $S^{2d_0+2d_\infty+3}/\bbz_l$ where $l$ is determined by the restriction of the K\"ahler form $\gro$ on $N^{ad}$ to the fibers. The action of $\bbz_l$ on $S^{2d_0+2d_\infty+3}$ is given by 
$$(z_1,\ldots, z_{d_0+d_\infty+2})\mapsto (\grl z_1,\ldots,\grl z_{d_0+d_\infty+2})$$ 
where $\grl$ is an $l$th root of unity.
\end{proof}

We have

\begin{theorem}\label{Mcoh}
Let $M$ be Boothby-Wang bundle over an admissible projective bundle $N^{ad}=\bbp(E_0\oplus E_\infty)\lra N$. Then
\begin{enumerate}
\item $\pi_1(N^{ad})=\pi_1(N)$;
\item $\pi_1(M)$ is an extension of $\pi_1(N)$ by $\bbz_l$;
\item there is a spectral sequence converging to $H^*(M,\bbz)$ with $E_2$ term
$$E^{p,q}_2=H^p\bigl(N,\calh^q(S^{2d_0+2d_\infty+3}/\bbz_l,\bbz)\bigr)$$
where $\calh^q$ is the derived functor sheaf.
Furthermore, if $N$ is simply connected then $E^{p,q}_2=H^p(N,H^q)$ where
$$H^q=
\begin{cases}
\bbz & \text{if $q=0$} \\
\bbz_l & \text{if $q=2,\ldots,2d_0+2d_\infty+2$} ~. \\
\bbz  & \text{if $q=2d_0+2d_\infty+3$} \\
 0   & \text{otherwise}
\end{cases}$$ 
\end{enumerate}
\end{theorem}

\begin{proof}
Items (1) and (2) follow from the exact homotopy sequences of the bundles in Proposition \ref{BWsphere} and its proof. Item (3) follows from the Leray spectral sequence of the oriented lens space bundles of Proposition \ref{BWsphere} together with the well known cohomology of lens spaces \cite{BoTu82}. 
\end{proof}

If we consider the cohomology with $\bbq$ coefficients and the dimension of $N$ is less than $2d_0+2d_\infty+4$, the spectral sequence collapses and we get

\begin{corollary}\label{Mrat}
If $\dim_\bbc N<d_0+d_\infty+2$, the Sasaki manifold $M$ has the rational cohomology groups of the product $S^{2d_0+2d_\infty+3}\times N$.
\end{corollary}

\begin{example}[continuation of Example \ref{V1blowdownexample}]
We want to describe the topology of our Sasaki manifold $M$ in Example \ref{V1blowdownexample}. We have

\begin{lemma}\label{cohN}
The cohomology of the 6-manifold $N=V_1$ is
$$H^q(N,\bbz)
\begin{cases}
\bbz & \text{if $q=0,2,4,6$} \\
\bbz^{42} & \text{if $q=3$} ~. \\
 0   & \text{otherwise}
\end{cases}$$ 

\end{lemma}

\begin{proof}
First we note that since $N$ is a smooth hypersurface in a weighted projective space, its cohomology is torsionfree \cite{Dim92}, and by the Lefschetz hyperplane theorem it has the cohomology of projective space for $q\neq 3$. For $q=3$ we can compute the Alexander polynomial of the Brieskorn 7-manifold $M$ over $N$ by the Milnor-Orlik method described in Section 9.3 of \cite{BG05}. This gives $H^3(M,\bbq)=\bbq^{42}$. But then an easy spectral sequence or Gysin sequence argument gives $H^3(N,\bbq)\approx H^3(M,\bbq)$ which gives the result since the cohomology of $N$ is torsionfree.
\end{proof}

Now continuing Example \ref{V1blowdownexample} we have $d_0=0$, and $\dim_\bbc N=3$. We shall assume that $d_\infty>0$, since the no blowdown case gives rise to a join which has been studied elsewhere \cite{BoTo14a,BoTo19a}. Moreover, if $d_\infty>1$ we can apply Corollary \ref{Mrat} which implies that $H^*(M,\bbq)\approx H^*(S^{2d_\infty +3}\times N,\bbq)$. If we can take $l=1$ in the proof of Proposition \ref{BWsphere}, then $M$ will be an $S^{2d_\infty +3}$ bundle over $N$. So in this case when $d_\infty>1$ the integral  cohomology spectral sequence of the bundle collapses giving the cohomology groups of the product $S^{2d_\infty +3}\times N$. 

 As briefly indicated at the end of Section \ref{cscsect}, there is a constraint on the index of $N$. It follows from Theorem 3.1 in \cite{ACGT08b} that, in order for $N^{ad}$ to be Fano, we need the index $\cali_N$ of $N$ to be at least $d_0+2$, so we can choose the K\"ahler form $\omega_1$ on $N$ in a way that its scalar curvature satisfies $s_1>d_0+1$. Thus, we take $N$ to be any of the families given by $1.11,1.12,1.13,1.16$ in the big table of \cite{Fanos}. These are smooth hypersurfaces in weighted projective spaces. In all these cases the cohomology of $N$ is torsionfree, and can differ from the cohomology of $\bbc\bbp^3$ only in the middle dimension. The 3rd Betti number can be computed using the Alexander polynomial as mentioned above. One can check the big table in Section 6 of \cite{Fanos} that these projective varieties $N$ are all polystable, and thus admit a KE metric. We indicate this in the following table of smooth Fano 3-folds.

\begin{center}
\begin{tabular}{ | c | c | c | c | c | c | c | c |}
\hline
$d$ &$\qquad \bfa$ & $ \bfw$ & $\cali_N$ & $H^3(N,\bbz)$ & KE  \\ \hline
 2 &$(2,2,2,2,2)$ & $(1,1,1,1,1)$&3 & $0$  & yes  \\ \hline
 3 &$(3,3,3,3,3)$ & $(1,1,1,1,1)$ &2& $\bbz^{10}$ & yes \\ \hline
4 & $(4,4,4,4,2)$ & $(1,1,1,1,2)$ &2 & $\bbz^{20}$ & yes \\ \hline
6 & $(6,6,6,3,2)$  & $(1,1,1,2,3)$ &2 & $\bbz^{42} $& yes \\ \hline
\end{tabular}
\end{center}
\bigskip

\begin{remark}
Applying these $N$ to Corollary \ref{SEexistence} and using Proposition \ref{BWsphere} proves the existence of SE metrics on certain lens space bundles over these smooth Fano 3-folds. 
\end{remark}

\end{example}

\appendix

\section{Constant weighted scalar curvature on admissible K\"ahler manifolds}

In this appendix we will show how the technique we used in the proof of Theorem \ref{CSCexistence}  can be adapted to 
an existence result for K\"ahler metrics of constant weighted scalar curvature. The notion of {\em constant weighted scalar curvature} was introduced more generally by A. Lahdili
in \cite{Lah19}, see also \cite{Ino22}.

Suppose that $\Omega$ is any admissible K\"ahler class on an admissible manifold $N^{ad}$ which fibers over a local product\footnote{See Definition \ref{adkahman} for what we mean by ``local product'' here.} of nonnegative CSCK metrics. 
We will show that when the value of the weight is sufficiently large, there exists an admissible K\"ahler metric in $\Omega$ with constant weighted scalar curvature. Specifically we prove the following theorem:

\begin{theorem}\label{WCSCexistence}
Suppose $\Omega$ is any admissible K\"ahler class on the admissible manifold $N^{ad}=\bbp(E_0 \oplus E_{\infty}) \lra N$, where $N$ is a compact K\"ahler manifold which is a local product of nonnegative CSCK metrics. If $p > \max\{m+1,2d_0+2,2d_\infty +2\}$, then there always exist some $c\in (-1,1)$ such that the corresponding admissible $(c\,\gz+1,p)$-extremal metric in $\Omega$ has constant $(c\, \gz+1,p)$-scalar curvature.
\end{theorem}

Note that Lemmas \ref{alphatrend} and \ref{betatrend} below supply additional technical details for the proof of Theorem \ref{CSCexistence}. In Section \ref{but} we will touch on the limitations of Theorem \ref{WCSCexistence}, by exploring an example that falls outside of the realm of Theorem \ref{WCSCexistence}.

\begin{proof}

Assume that we have an admissible manifold $N^{ad}$ as above with the property that $N$ is a local product of {\em non-negative} CSCK metrics and let $\Omega$ be an admissible K\"ahler class on $N^{ad}$. As we discussed 
in the beginning of Section \ref{E&NE}, we know that $\Omega$ admits an admissible $(c\,\gz+1,p)$-extremal metric for any choice of $c\in (-1,1)$ and $p\in \bbr$. 
Recall that the admissible $(c\,\gz+1,p)$-extremal metric has $(c\,\gz+1,p)$-scalar curvature $Scal_{c\,\gz +1,p}=A_1\gz+A_2$, where, $A_1$ and $A_2$ are given \eqref{A1andA2}
and \eqref{alphabeta}. 
Thus admissible $(c\,\gz+1,p)$-extremal metric has, $g$ has constant weighted scalar curvature, $Scal_{c\,\gz +1,p}$, if and only if 
$A_1=0$. In turn the equation $A_1=0$ may be re-written as    
\begin{equation}\label{cweighted}       
\alpha_{1,-(1+p)}\beta_{0,(1-p)}-\alpha_{0,-(1+p)}\beta_{1,(1-p)}=0.                
\end{equation}

As already mentioned in the proof of Theorem \ref{CSCexistence} we know that 
\begin{itemize}
\item $c\, \alpha_{r+1,k}+\alpha_{r,k}=\alpha_{r,k+1}$
\item $c\, \beta_{r+1,k}+\beta_{r,k}=\beta_{r,k+1}$.
\end{itemize}
Further, it is easy to check the general formulas below:
\begin{itemize}
\item $\frac{d}{dc}\left[\alpha_{r,k}\right]=k\alpha_{r+1,k-1}$
\item  $\frac{d}{dc}\left[\beta_{r,k}\right]=k \beta_{r+1,k-1}$,
\end{itemize}
These observations are used in the calculations below.

We now turn our attention to the {\em weighted Einstein-Hilbert functional} as defined by Futaki and Ono in \cite{FutOno20}:
In the present case, (the appropriate restriction of) this weighted Einstein-Hilbert functional is given by
$$WH_K(c):=\frac{S_c}{V_c^{\frac{p-2}{p}}},$$
where 
$V_c= \int_{N^{ad}}(c\,\gz+1)^{-p}\,d\mu_g$ and $S_c= \int_{N^{ad}} Scal_{c\,\gz+1,p}(g) (c\,\gz+1)^{-p}\,d\mu_g$. Using the formulas from
\cite{ApMaTF18}, it is not hard to check that, up to an overall positive rescale, $WH_K(c)= \frac{\beta_{0,2-p}}{\left(\alpha_{0,-p}\right)^{\frac{p-2}{p}}}$.
To simplify things a bit, we will actually work with $H_K(c):=[WH_K(c)]^p$, that is,
$$H_K(c)= \frac{\left(\beta_{0,2-p}\right)^{p}}{\left(\alpha_{0,-p}\right)^{p-2}}.$$ Note that while $H_S(c)$ from the proof of Theorem \ref{CSCexistence} is analogous to $H_K(c)$, it is NOT a special case of $H_K(c)$ and for $p=m+2$, the two functions are not equal.
We easily verify that 
$$
\begin{array}{ccl}
H_K'(c) & = & \frac{p(2-p)\left(\alpha_{0,-p}\right)^{p-2}\left(\beta_{0,2-p}\right)^{(p-1)}\beta_{1,1-p} - (p-2)(-p)(\alpha_{0,-p})^{(p-3)}\alpha_{1,-(1+p)}\left(\beta_{0,2-p} \right)^{p}}{\left(\alpha_{0,-p}\right)^{2(p-2)}}\\
\\
&= &\frac{p(p-2)(\beta_{0,2-p})^{p-1} (\alpha_{0,-p})^{p-3}}{\left(\alpha_{0,-p}\right)^{2(p-2)}}\left[ \alpha_{1,-(1+p)}\beta_{0,2-p}-\alpha_{0,-p}\beta_{1,1-p} \right]\\
\\
&= &\frac{p(p-2)(\beta_{0,2-p})^{p-1} (\alpha_{0,-p})^{p-3}}{\left(\alpha_{0,-p}\right)^{2(p-2)}}\left[\alpha_{1,-(1+p)}\beta_{0,(1-p)}-\alpha_{0,-(1+p)}\beta_{1,(1-p)}\right].
\end{array}
$$
Now $\alpha_{0,-p}>0$ and $\beta_{0,2-p}> 0$ (the latter is due to the fact that we are assuming $N$ is a local product of {\em nonnegative} CSCK metrics) and therefore critical
points of $H_K(c)$ correspond exactly to solutions of Equation \eqref{cweighted}. Showing that for $p > \max\{m+1,2d_0+2,2d_\infty +2\}$, $H_K(c)$ must have a least one critical point $c\in(-1,1)$, would finish the proof.

To that end, first notice that $H_K(c)$ is a smooth positive function for $c\in (-1,1)$. To show that $H(c)$ has a critical point inside $(-1,1)$, we will show that $\displaystyle \lim_{c\rightarrow \pm 1^{\mp}}H_K(c) = +\infty$. We need two lemmas. Let $o(x)$ denote any function which is smooth in a neighborhood of $x=0$ and satisfies that $o(0)=0$.
We shall recycle this notation for several functions of this nature.

\begin{lemma}\label{alphatrend}
For any $k \geq m+1$,
$$\alpha_{0,-k}=\displaystyle \frac{\delta_0(c) + o(1-c)}{(1-c)^{k-1-d_0}}=\frac{\delta_\infty(c) + o(1+c)}{(1+c)^{k-1-d_\infty}},$$
where $\delta_0$ and $\delta_\infty$ are smooth functions defined in open neighborhoods of $c=1$ and $c=-1$ respectively, such that 
$\delta_0(1)>0$ and $\delta_\infty(-1)>0$.
\end{lemma}

\begin{proof}
By repeated integration by parts, we can see that $\alpha_{0,-k}=  \int_{-1}^1 (c\,t +1)^{-k} \overbrace{p_c(t)}^{\text{degree}\,m-1}dt$ is a rational function of $c$ which is
positive and has no vertical asymptotes within the interval $(-1,1)$. 

First notice that 
$$\alpha_{0,-k} =  \int_{-1}^1 (c\,t +1)^{-k} (1+t)^{d_0} p_0(t)dt =  \int_{-1}^1 (c\,t +1)^{-k} (1-t)^{d_\infty} p_{\infty}(t)dt,$$
where $p_0(t) = p_c(t)/(1+t)^{d_0}$ and $p_\infty(t) = p_c(t)/(1-t)^{d_\infty}$ are polynomials of degree $m-1-d_0$ and $m-1-d_\infty$, respectively. Keep in mind that
the complex dimension $m=\sum_{a\in{\hat\cala}}d_a+1$ (with $d_a>0$ for $a\in \cala$ and $d_a\geq 0$ for $a\in \hat\cala$) and note that $p_0(-1)>0$ and $p_{\infty}(1)>0$. From the definition of $p_c(t)$ we also know that 
$p_0(t)$ and $p_\infty(t)$ are positive for $-1<t<1$. To understand the behavior of $\alpha_{0,-k}$ near $c=\pm 1$, we consider repeated integration by parts and observe that
$$\alpha_{0,-k}=\displaystyle \frac{\left( \frac{(k-2-d_0)! \, d_0! \, p_0(-1)}{(k-1)! \, c^{d_0+1}} + o(1-c)\right)}{(1-c)^{k-1-d_0}}=\frac{\left( \frac{(-1)^{d_\infty+1}(k-2-d_\infty)! \, d_\infty! \, p_\infty(1)}{(k-1)! \, c^{d_\infty+1}} + o(1+c)\right)}{(1+c)^{k-1-d_\infty}}$$
This proves the lemma.
\end{proof}

\begin{lemma}\label{betatrend}
For any $l \geq m$,
$$\beta_{0,-l}=\displaystyle \frac{\gamma_0(c) + o(1-c)}{(1-c)^{l-d_0}}=\frac{\gamma_\infty(c) + o(1+c)}{(1+c)^{l-d_\infty}},$$
where $\gamma_0$ and $\gamma_\infty$ are smooth functions defined in open neighborhoods of $c=1$ and $c=-1$ respectively, such that 
$\gamma_0(1)>0$ and $\gamma_\infty(-1)>0$.
\end{lemma}

\begin{proof}
By repeated integration by parts, we can see that 
$$
\begin{array}{ccl} 
\beta_{0,-l} &= & \int_{-1}^1\overbrace{\Big(\sum_{a\in {\hat \cala}} \frac{x_ad_as_a}{1+x_at}\Big)  p_c(t)}^{\text{degree}\, m-2}  (c\,t +1)^{-l} dt \\
\\
& + & \big((1-c)^{-l}p_c(-1) + (1+c)^{-l} p_c(1)\big)
      \end{array}
$$
is a rational function of $c$ which is
non-negative and has no vertical asymptotes within the interval $(-1,1)$. 
Moreover, it is easy to see that if $d_0=0$,
$$\beta_{0,-l} = \frac{\left( p_c(-1)+o(1-c)\right)}{(1-c)^l},$$
with $p_c(-1)>0$ in this case,
and  if $d_\infty=0$,
$$\beta_{0,-l} = \frac{\left( p_c(1)+o(1+c)\right)}{(1-c)^l}$$
with $p_c(1)>0$ in this case.

Now assume that $d_0 \geq 1$ and define
$$q_0(t) := \sum_{a\in {\hat \cala}\setminus \{0\} } \frac{x_ad_as_a}{1+x_at}\frac{p_c(t)}{(1+t)^{d_0-1}} + x_0 d_0 s_0 \frac{p_c(t)}{(1+t)^{d_0}}.$$
The degree of the polynomial $q_0(t)$ is $m-d_0-1$ and 
$$\displaystyle q_0(-1) =  x_0 d_0 s_0 \prod_{a\in {\hat \cala}\setminus \{0\} } (1-x_a)^{d_a} =d_0 (d_0+1) \prod_{a\in {\hat \cala}\setminus \{0\} } (1-x_a)^{d_a} >0.$$ 
By repeated integration by parts we then have
$$
\begin{array}{ccl} 
\beta_{0,-l} &= & \int_{-1}^1 (c\,t +1)^{-l} (1+t)^{d_0-1} q_0(t) dt + (1+c)^{-l} p_c(1) \\
\\
& =&    \displaystyle \frac{\left( \frac{(l-2-(d_0-1))! \, (d_0-1)! \, q_0(-1)}{(l-1)! \, c^{d_0}} + o(1-c)\right)}{(1-c)^{l-d_0}}.   \end{array}
$$

Similarly if $d_\infty \geq 1$ we define
$$q_\infty(t) := \sum_{a\in {\hat \cala}\setminus \{\infty\} } \frac{x_ad_as_a}{1+x_at}\frac{p_c(t)}{(1-t)^{d_\infty-1}} + x_\infty d_\infty s_\infty \frac{p_c(t)}{(1-t)^{d_\infty}}.$$
The degree of the polynomial $q_\infty(t)$ is $m-d_\infty-1$ and 
$$\displaystyle q_\infty(1) =  x_\infty d_\infty s_\infty \prod_{a\in {\hat \cala}\setminus \{\infty\} } (1+x_a)^{d_a} =d_\infty (d_\infty+1) \prod_{a\in {\hat \cala}\setminus \{\infty\} } (1+x_a)^{d_a} >0.$$
By repeated integration by parts we now have
$$
\begin{array}{ccl} 
\beta_{0,-l} &= & \int_{-1}^1 (c\,t +1)^{-l} (1-t)^{d_\infty-1} q_\infty(t) dt + (1-c)^{-l} p_c(-1) \\
\\
& =&    \displaystyle \frac{\left( \frac{(-1)^{d_\infty}(l-2-(d_\infty-1))! \, (d_\infty-1)! \, q_\infty(1)}{(l-1)! \, c^{d_\infty}} + o(c+1)\right)}{(c+1)^{l-d_\infty}}.   \end{array}
$$
From the above observations we have now proved Lemma \ref{betatrend} regardless of the values of $d_0$ and $d_\infty$.
\end{proof}

From Lemmas \ref{alphatrend} and \ref{betatrend} we now have that for $k \geq m+1$ and $l\geq m$,
$$ \frac{\left(\beta_{0,-l}\right)^{k}}{\left(\alpha_{0,-k}\right)^{l}}= \frac{\left(\gamma_0(c) + o(1-c)\right)^k}{\left(\delta_0(c)+o(1-c)\right)^l}\left(1-c\right)^{d_0(k-l)-l}
= \frac{\left(\gamma_\infty(c) + o(1+c)\right)^k}{\left(\delta_\infty(c)+o(1+c)\right)^l}\left(1+c\right)^{d_\infty(k-l)-l}.$$
In particular, for $p\geq m+2$, we see that (letting $l=p-2$ and $k=p$),

$$H_K(c)= \frac{\left(\beta_{0,2-p}\right)^{p}}{\left(\alpha_{0,-p}\right)^{p-2}}=\frac{\left(\gamma_0(c) + o(1-c)\right)^p}{\left(\delta_0(c)+o(1-c)\right)^{p-2}}\left(1-c\right)^{2d_0 +2-p}
= \frac{\left(\gamma_\infty(c) + o(1+c)\right)^p}{\left(\delta_\infty(c)+o(1+c)\right)^{p-2}}\left(1+c\right)^{2d_\infty+2-p}.$$

Therefore, for $p > \max\{m+1,2d_0+2,2d_\infty +2\}$, $\displaystyle \lim_{c\rightarrow \pm 1^{\mp}}H_K(c) = +\infty$. This, together with the fact that $H_K(c)$ is smooth and bounded from below over the interval $(-1,1)$, allows us to conclude that $H_K(c)$ has at least one critical point $c\in (-1,1)$. For this value of $c$, we have $H_K'(c)=0$ and therefore $c$ solves \eqref{cweighted}. 

\end{proof}

Recall that $\displaystyle m=\sum_{a\in{\hat\cala}}d_a+1$. Assuming (as we should) that $\cala \neq \emptyset$, we have that for $p=2m$ the condition
$p > \max\{m+1,2d_0+2,2d_\infty +2\}$ is automatic. As a special case, we thus get the following corollary to Theorem \ref{WCSCexistence}, which in particular confirms Conjecture 1 in \cite{ApMaTF18}. Note that the corollary is also a special case of an existence theorem stated in \cite{Gua23}.

\begin{corollary}\label{EMexistence}
Suppose $\Omega$ is any admissible K\"ahler class on the admissible manifold $N^{ad}=\bbp(E_0 \oplus E_{\infty}) \lra N$, where $N$ is a compact K\"ahler manifold which is a local product\footnote{See Definition \ref{adkahman} for what we mean by ``local product'' here.} of nonnegative CSCK metrics. Then there always exist some $c\in (-1,1)$ such that the corresponding admissible $(c\,\gz+1,2m)$-extremal metric $g$ in $\Omega$ 
has the property that $h=(c\,\gz+1)^{-2}g$ has constant scalar curvature and therefore is conformally K\"ahler Einstein-Maxwell.
\end{corollary}

\subsection{Limitation of Theorem \ref{WCSCexistence}}\label{but}

On a more pessimistic note, if we instead work with $p=m+2$, we cannot assume that $p > \max\{m+1,2d_0+2,2d_\infty +2\}$ is true and thus we cannot appeal to
Theorem \ref{WCSCexistence}. The following example illustrates that, in this case, existence is more unpredictable - even if the underlying admissible manifold is fixed.

\begin{example}
Let $\cala=\{1\}$, $N=N_1=\bbc\bbp^1$, and $s_1=2$ (so $g_1$ is the Fubini-Study metric on $\bbc\bbp^1$). If we then let $d_\infty =1$ and $d_0=3$, we have $m=6$ and $p=8$. Clearly $p=2d_0+2$ and hence $p > \max\{m+1,2d_0+2,2d_\infty +2\}$ is false. Setting $x_1=x$ we have that
each value $0<x<1$ determines a certain admissible K\"ahler class. We will see how the existence of a solution $c\in (-1,1)$ to \eqref{cweighted} depends on the value of
$x$.

We calculate that
$$
\begin{array}{ccl}
f_x(c)&:=& \alpha_{1,-9}\beta_{0,-7}-\alpha_{0,-9}\beta_{1,-7}\\
\\
&=& -12 x + (24 + x + 13 x^2)c+ (-17 - 52 x + x^2)c^2+(-1 + 15 x + 28 x^2)c^3.
\end{array}
$$
Thus solutions $c\in (-1,1)$ to \eqref{cweighted} correspond to roots in $(-1,1)$ of the cubic $f_x(c)$.
This cubic (with real coefficients) has discriminant $D_x$ given by
$$
D_x=-5 (1 + x)^3 (-44352 + 155904 x - 159125 x^2 + 115249 x^3 - 
   90591 x^4 + 49179 x^5).
   $$
We can check numerically that there exists a specific value $\tilde{x} \in (0,1)$ such that $D_{\tilde{x}}=0$ and
for $0<x<\tilde{x}$, $D_x>0$, while for $\tilde{x}<x<1$, $D_x<0$. Note that $\tilde{x} \approx 0.429$.

For $\tilde{x}<x<1$, we know that the cubic
$f_x(c)$ has exactly one real root $\hat{c}$. One can check that for this range of $x$ values, $f_x(1)<1$ and $\displaystyle\lim_{c\rightarrow +\infty}f_x(c)=+\infty$, so it is clear that $\hat{c}>1$ and thus for $\tilde{x}<x<1$, $f_x(c)$ has no roots in $(-1,1)$. Therefore, the admissible K\"ahler classes corresponding to values $\tilde{x}<x<1$ admit no
admissible K\"ahler metrics with constant $(c\,\gz+1,8)$-scalar curvature. 
The question remains whether there could be any K\"ahler metric in those K\"ahler classes with
constant $(f,8)$-scalar curvature, where $f$ is some positive killing potential.

For $0<x< \tilde{x}$, we know that the cubic $f_x(c)$ has three distinct real roots.  We will see that at least one of them will be inside $(-1,1)$.
\begin{itemize}
\item In the case where $0<x< 1/7$, this is completely straightforward since $f_x(-1)=-40 (1 + x)^2<0$ and $f_x(1)=6 (1 - x) (1 - 7 x)>0$. 
\item In the case where $x=1/7$, we have $f_{1/7}(c) = \frac{4}{49} (c-1) \left(21 c^2-278 c+21\right)$ and it is easy to check that
$21 c^2-278 c+21$ has a root in $(-1,1)$.
\item In the case where $1/7<x<\tilde{x}$, we have $f_x(-1)<0$ and $f_x(1)<0$, but we also have that 
$f_x'(-1) =5 (1 + x) (11 + 19 x)>0$, $f_x'(1)=-13 - 58 x + 99 x^2 <0$, and $\displaystyle\lim_{c\rightarrow +\infty}f_x(c)=+\infty$.
Considering the options for the cubic with three distinct real roots, it is clear that two of the roots must be inside $(-1,1)$ and that the third
root is in the interval $(1,+\infty)$.
\end{itemize}

Finally, for $x=\tilde{x}$, we know that the cubic $f_x(c)$ has a multiple root (and all the roots are real). In this case we still have
that $f_{\tilde{x}}(-1)<0$, $f_{\tilde{x}}'(-1)>0$, $f_{\tilde{x}}(1)<0$, $f_{\tilde{x}}'(1)<0$, and $\displaystyle\lim_{c\rightarrow +\infty}f_x(c)=+\infty$.
It is therefore clear that in this case $f_{\tilde{x}}(c)$ has a double real root in the interval
$(-1,1)$ and a single real root in the interval $(1,+\infty)$.
In conclusion, for each admissible K\"ahler class, $\Omega_x$, corresponding to a value $0<x \leq \tilde{x}$, there exists at least one value $c \in (-1,1)$ such that
$\Omega_x$ admits an admissible K\"ahler metric with constant $(c\,\gz+1,8)$-scalar curvature.

\end{example}

\bibliography{ref}
\bibliographystyle{amsalpha} 

\end{document}